\documentclass[12pt]{amsart}
\usepackage{amscd,amssymb,graphicx,color,a4wide,hyperref,cite,amsthm,amsmath}
\usepackage{listings}
\usepackage{color}
\usepackage{pbox}
\definecolor{dkgreen}{rgb}{0,0.6,0}
\definecolor{gray}{rgb}{0.5,0.5,0.5}
\definecolor{mauve}{rgb}{0.58,0,0.82}

\definecolor{amaranth}{rgb}{0.9, 0.17, 0.31}
\usepackage{hyperref}
\hypersetup{colorlinks=true,citecolor=amaranth,linkcolor=black}
\usepackage{url}

\usepackage{mathrsfs}

\usepackage{epstopdf}

\usepackage[all]{xy}
\let\oldtocsection=\tocsection
 
\let\oldtocsubsection=\tocsubsection
 
\let\oldtocsubsubsection=\tocsubsubsection
 
\renewcommand{\tocsection}[2]{\hspace{0em}\oldtocsection{#1}{#2}}
\renewcommand{\tocsubsection}[2]{\hspace{1em}\oldtocsubsection{#1}{#2}}
\renewcommand{\tocsubsubsection}[2]{\hspace{2em}\oldtocsubsubsection{#1}{#2}}

\newtheorem{theorem}{Theorem}[section]

\newtheorem{proposition}[theorem]{Proposition}

\theoremstyle{definition}
\newtheorem{definition}[theorem]{Definition}

\newtheorem{example}[theorem]{Example}

\usepackage{amsthm}
\theoremstyle{remark}
\newtheorem{remark}[theorem]{Remark}

\newtheorem*{notation}{Notation}

\numberwithin{equation}{section}
\numberwithin{figure}{section}

\newcommand{\bbfamily}{\fontencoding{U}\fontfamily{bbold}\selectfont}
\newcommand{\textbb}[1]{{\bbfamily#1}}

\newcommand {\lfor} {\mbox{\textbb{[}}}
\newcommand {\rfor} {\mbox{\textbb{]}}}

\newcommand{\shP} {\mathcal{P}}

\newcommand{\ZZ} {\mathbb{Z}}

\newcommand{\RR} {\mathbb{R}}
\newcommand{\CC} {\mathbb{C}}

\renewcommand{\AA} {\mathbb{A}}

\newcommand {\lra} {\longrightarrow}

\newcommand\restr[2]{{
  \left.\kern-\nulldelimiterspace 
  #1 
  \vphantom{\big|} 
  \right|_{#2} 
  }}

\newcommand {\ul} {\underline}
\newcommand {\ol} {\overline}
  
\usepackage{tikz}
\usepackage{MnSymbol}
\usetikzlibrary{matrix,arrows,decorations.pathmorphing}
\usepackage{hyperref}
\usepackage{enumerate}
\usepackage{float}
\usepackage{tikz-cd}
\usepackage{relsize}
\usepackage{mathtools}

\newcommand\C{\mathbb C}

\newcommand\NN{\mathbb N}

\newcommand\PP{\mathbb P}
\newcommand\R{\mathbb R}

\newcommand\Z{\mathbb Z}

\newcommand\cP{\mathcal P}

\newcommand\Spec{\operatorname{Spec}\,}

\newcommand\Hom{\operatorname{Hom}}

\newcommand {\out}  {\mathrm{out}}
\newcommand{\foD} {\mathfrak{D}}
\newcommand{\fod} {\mathfrak{d}}
\newcommand {\fom}  {\mathfrak{m}}
\newcommand {\Sing} {\operatorname{Sing}}

\newtheoremstyle{cited}%
  {3pt}
  {3pt}
  {\itshape}
  {}
  {\bfseries}
  {.}
  {.5em}
  {\thmname{#1} \thmnumber{#2} \thmnote{\normalfont#3}}
\theoremstyle{cited}
\newtheorem{citedthm}{Theorem}

\newcommand {\gp}  {{\operatorname{gp}}}
\newcommand {\scrP}  {\mathscr{P}}
\renewcommand{\P}  {\mathscr{P}}
\newcommand {\Div}  {\operatorname{Div}}

\newcommand{\heart}{\ensuremath\heartsuit}


\usepackage[tableposition=above]{caption}

\begin{document}

\title[The heart of canonical wall structures]{
Equations of mirrors to log Calabi--Yau pairs \\
\small{via the heart of canonical wall structures} \\ }
\author{H\"ulya Arg\"uz}
\address{University of Georgia, Dept of Mathematics, 
Athens, GA 30602}
\email{Hulya.Arguz@uga.edu}

\date{\today}

\maketitle
\vspace{-0.5cm}
\begin{abstract}
Gross and Siebert developed a program for constructing in arbitrary dimension a mirror family to a log Calabi--Yau pair $(X,D)$, consisting of a smooth projective variety $X$ with a normal-crossing anti-canonical divisor $D$ in $X$. 
In this paper, we provide an algorithm to practically compute explicit equations of the mirror family in the case when $X$ is obtained as a blow-up of a toric variety along hypersurfaces in its toric boundary, and $D$ is the strict transform of the toric boundary. The main ingredient is \emph{the heart of the canonical wall structure} associated to such pairs $(X,D)$, which is constructed purely combinatorially, following our previous work with Mark Gross. 
In the case when we blow up a single hypersurface we show that our results agree with previous results computed symplectically by Aroux--Abouzaid--Katzarkov. In the situation when the locus of blow-up is formed by more than a single hypersurface, due to infinitely many walls interacting, writing the equations becomes significantly more challenging. We provide the first examples of explicit equations for mirror families in such situations.

\end{abstract}

\setcounter{tocdepth}{2}
\tableofcontents

\section{Introduction}
\label{introduction}

\subsection{Overview}

Gross and Siebert developed a program for constructing mirror families to Calabi--Yau varieties algebro-geometrically \cite{GSCanScat}. More recently, this construction has been extended to the set up of log Calabi--Yau pairs $(X,D)$, given by a smooth projective variety $X$ along with a reduced normal--crossings anticanonical divisor $D$. 
The construction of the mirror family to $(X,D)$ -- or rather to the complement $X \setminus D$ -- uses a
\emph{canonical wall structure} on an affine manifold with singularities arising as the tropicalization of $(X,D)$ \cite{GSCanScat}. Roughly put, such a structure is a combinatorial gadget incorporating tropical analogues of all rational stable (log) maps to $(X,D)$, with a specified marked point mapping to $D$. 
Such maps, referred to as \emph{$\AA^1$-curves} throughout this paper, give rise to well defined invariants of $(X,D)$, and fit into the more general framework of punctured log Gromov--Witten invariants defined by Abramovich--Chen--Gross--Siebert \cite{ACGSI,ACGSII}. 

For a toric log Calabi--Yau pair $(X_{\Sigma},D_{\Sigma})$, given by a smooth toric variety $X_{\Sigma}$ associated to a complete fan $\Sigma$ in $\RR^n$, along with the toric boundary divisor $D_{\Sigma}$, the construction of the mirror family is pretty straightforward as there are no $\AA^1$-curves in $(X_{\Sigma},D_{\Sigma})$ -- any curve in a toric variety touching the boundary at one point would necessarily touch also other boundary components. Thus, toric log Calabi--Yau pairs $(X_{\Sigma},D_{\Sigma})$ form 
form an immediate class of examples where we know how to write explicit equations for the mirror family. Beyond this, so far there have been very few examples of explicit equations of mirrors. Particularly, in dimension two explicit equations for mirror families to few log Calabi--Yau surfaces surfaces could be computed using computer algebra \cite{barrott2018explicit}, and in dimension three only in one case, a three dimensional analogue of the del Pezzo surface of degree $7$, the mirror is understood \cite{ducat20213}. 

A particular challenge to compute equations of mirror families to log Calabi--Yau pairs $(X,D)$ in bigger generality arises due to the fact
computing counts of $\AA^1$-curves which appear in the construction of the canonical wall structure is technically difficult. In our joint with Mark Gross \cite{AG}, generalizing previous results of Gross--Pandharipande--Siebert \cite{GPS} in dimension two to higher dimensions, we show that for particular log Calabi--Yau pairs $(X,D)$, there is a purely algebraic algorithm to capture the data of $\AA^1$-curves appearing in the construction of the canonical wall structure. Such a log Calabi--Yau pair $(X,D)$, which we study in \cite{AG}, is given by a blow-up
\begin{equation}
\label{Eq: blow up}
    X \longrightarrow X_{\Sigma}
\end{equation}
of a toric log Calabi--Yau pair $(X_\Sigma, D_\Sigma)$ along hypersurfaces of the toric boundary $D_\Sigma$, and where $D$ is the strict transform of $D_{\Sigma}$,  
The algebraic algorithm giving the counts of $\AA^1$-curves of such a pair uses a degeneration of $X$ into the union of the toric variety $X_{\Sigma}$ and some simpler components obtained as blow-ups of $\PP^1$ bundles over the toric boundary. Working with such a degeneration enables us to reduce the complicated enumerative geometry of $(X,D)$ to a toric situation, which amounts to pulling singularities out from the canonical wall structure and working with a simpler wall structure in $\RR^n$. In this paper, we describe a wall structure associated to a log Calabi--Yau pair $(X,D)$, obtained from this simpler wall structure in $\RR^n$, by eliminating from it all classes of curves which are not in $X$. We then show that the resulting wall structure, which we call the \emph{heart of the canonical wall structure associated to $(X,D)$}, produces the correct mirror family as in \cite{GSCanScat}.

The advantage of working with the heart of the canonical wall structure is that it is constructed purely combinatorially, and thus provides a combinatorial recipe to write explicitly equations for mirror families. As a particular application, we write explicit equations of mirrors to three dimensional log Calabi--Yau pairs obtained by non-toric blow-ups of $\PP^3$
along unions of hypersurfaces contained in the toric boundary. This provides the first non-trivial examples of mirror families to log Calabi--Yau pairs in dimension bigger than two. 
In the situation where one considers the blow up of a toric variety along only a single hypersurface the mirror has been constructed earlier in the work of Aroux--Abouzaid--Katzarkov \cite{AAK} using symplectic geometric tools. We prove in \S \ref{Eq AAK mirror} that our mirror construction agrees with the one of \cite{AAK}, restricted to this situation.

\subsection{Background}
Associated to a log Calabi--Yau pair $(X,D)$ is its tropicalization, given by a polyhedral complex $B$ defined similarly as in the two dimensional case in \cite[\S1.2]{GHK}. This polyhedral complex carries the structure of an integral affine manifold with singularities, with singular locus $\Delta\subset B$. To define the canonical wall structure, one first fixes a submonoid $Q\subset N_1(X)$ containing all effective curve classes, where $N_1(X)$ denotes the abelian group generated by projective irreducible curves in $X$ modulo numerical equivalence \cite[Defn 1.8]{GHS}. The canonical wall structure associated to $(X,D)$ is then given by pairs
\[ \foD_{(X,D)} := \{(\fod,f_{\fod})\} \]
of codimension one subsets $\fod \subset B$ called walls, along with attached functions $f_{\fod}$, called wall-crossing functions, that are elements of the completion of $\mathbf{k}[\shP_x^+]$ at the ideal generated by $Q\setminus\{0\}$, where $\shP_x^+=\Lambda_x\times Q$, 
$x\in \mathrm{Int}{\fod}$ is a general point and $\Lambda$ is the local
system of integral vector fields on $B\setminus\Delta$. These functions $f_{\fod}$ are explicitly given by
\begin{equation}
\label{Eq: canonical wall}
    f_{\fod}= \exp(k_{\tau}N_{\tilde\tau}t^{\ul{\beta}} z^{-u})
\end{equation}
where $\tilde\tau=(\tau,\ul{\beta})$ ranges over types of $\dim X-2$-dimensional families
of tropicalizations of $\AA^1$-curves in $(X,D)$ of class $\ul{\beta} \in H_2(X,\ZZ)$. The contact order of the image of such an $\AA^1$-curve is tropically recorded in the tangent vector $u\in \Lambda_x$, for a general point $x\in \fod$, and $k_{\tau}$ is a positive integer depending
only on the tropical type $\tau$ as in \cite[\S2.4]{GSUtah} or 
\cite[(3.10)]{GSCanScat}. The term $t^{\ul{\beta}}z^{-u}$ denotes
the monomial in $\textbf{k}\lfor \Lambda_x \times Q \rfor$ associated to $(-u,\ul{\beta})$, and the number $N_{\tilde\tau}$ is an invariant of $(X,D)$, defined via counts of all $\AA^1$-curves of contact order $u$, and type $\tau$ \cite{ACGSI,ACGSII}.

A key result in our joint work with Mark Gross \cite{AG} shows that when $(X,D)$ is a log Calabi--Yau pair obtained as a blow-up of a toric log Calabi--Yau pair $(X_{\Sigma},D_{\Sigma})$ 
with center a union of general hypersurfaces of the toric boundary, the canonical wall structure can be constructed combinatorially, without using 
the enumerative invariants given by counts of $\AA^1$-curves. We do this by following \cite{GPS}, and considering a degeneration $(\widetilde X, \widetilde D)$ of $(X,D)$ obtained from a blow-up of the degeneration to the normal cone of $X_{\Sigma}$, with general fiber $(X,D)$. We then investigate the canonical wall structure associated to $(\widetilde X, \widetilde D)$, which has support in the tropicalization $\widetilde B$
of $(\widetilde X,\widetilde D)$. This tropicalization comes naturally
with a projection map $\widetilde p: \widetilde B \to \RR_{\geq 0}$.
Hence, we obtain a wall structure $\foD^1_{(\widetilde X, \widetilde D)}$ supported on $\widetilde B_1:= \widetilde p^{-1}(1)$, which is an integral affine manifold with singularities away from the origin. 
Localizing to the origin $0\in \widetilde B_1$ we obtain a wall structure 
\begin{equation}
\label{Eq T0 intro}
T_0\foD^{1}_{(\widetilde X,\widetilde D)}:=\{(T_0\fod, f_{\fod})\,|\,
(\fod,f_{\fod})\in \foD^{1}_{(\widetilde X,\widetilde D)}, \quad 0\in \fod\}
\end{equation}
on $T_0\widetilde B_1$, the tangent space to $0\in \widetilde B_1$. We then relate this wall structure via piecewise linear isomorphisms both to the canonical wall structure associated to $(X,D)$, and to a combinatorially constructed wall structure on $(\RR^n,\Sigma)$. In this paper, following \cite{AG} we take as a starting point the description of $T_0\foD^{1}_{(\widetilde X,\widetilde D)}$. By suitably modifying it to eliminate the classes of curves in $\widetilde X \setminus X$ which appear in the wall crossing functions, we construct the heart of the canonical wall structure. We give a more detailed overview of the construction and its consequences in what follows.

\subsection{Outline of the paper and main results}
One of the objectives of this paper is to provide readers who are not familiar with working with computations using wall-structures many examples, starting from easy ones going to technically involved ones. Therefore, we first review the construction of the coordinate ring $\mathcal{R}_{(X_{\Sigma},D_{\Sigma})}$ of a mirror family to a toric log Calabi--Yau pair $(X_{\Sigma},D_{\Sigma})$, by adopting the general construction of \cite{GSCanScat} to this primitive case where there are no walls -- or all walls carry trivial wall crossing functions given by identity. We explain how the associated ring $\mathcal{R}_{(X_{\Sigma},D_{\Sigma})}$ is generated by \emph{theta functions} (see \S\ref{Sec: mirrors to toric}). We then review the general construction of the theta functions generating the coordinate ring of the mirror to a log Calabi--Yau pair $(X,D)$ using \emph{broken lines} in the canonical wall structure (see \S \ref{sec:canonical scattering}). These are piecewise linear analogues of holomorphic discs on $(X,D)$, given by proper continuous maps $\beta : (-\infty, 0] \to B$
ending at $\beta(0)$, which carry monomials and allow us to trace how these monomials change each time the image of $\beta$ crosses a wall while approaching $\beta(0)$ (see \S\ref{sec: theta functions defined by broken lines}). In \S \ref{Sec: the heart of the canonical wall structure} we introduce the heart of the canonical wall structure, and prove our main result showing that the theta functions generating the mirror family to $(X,D)$ can be defined using broken lines in the heart. In the final section, as the main application of this construction, we compute the equations of the mirror in several three dimensional examples. In the remaining part we provide more details on the results we prove along the way. 

First recall that to describe the canonical wall structure associated to a log Calabi--Yau pair in $(X,D)$, one a priori fixes a monoid $Q \subset N_1(X)$, containing all effective curve classes, and the base of the corresponding mirror family is the formal completion of $\Spec \mathbf{k}[Q]$ at the maximal ideal $Q \setminus \{0\}$. However, it follows from the construction of the mirror family \cite{GHK,GSCanScat}, that it actually lives over a smaller base where $Q$
is replaced by the \emph{relevant monoid } $Q(X,D)$ defined as the set of integral points of the \emph{relevant cone of curves},
 \[\mathcal{C}(X,D) \subset N_1(X) \otimes \RR \]
generated by the union of $\AA^1$-curves in $(X,D)$ and curves in the boundary $D$, see Definition \ref{QXD}.

If $X$ is of dimension two, we show that the relevant cone of curves is simply the Mori cone of effective curves:

\begin{citedthm}[(=Theorem \ref{Theorem relevant Mori})]
Let $(X,D)$ be a generic log Calabi--Yau pair as in Definition \ref{Def: generic log CY}. Then, the 
relevant cone of curves $\mathcal{C}(X,D)$ in Definition \ref{QXD} is isomorphic to the Mori cone $\mathrm{NE}(X)$. 
\end{citedthm}
The generalization of  Theorem \ref{Theorem relevant Mori} to higher dimensions is wrong -- see Remark \ref{Rem: cant go higher} for a counterexample.

Now assume we are given a log Calabi--Yau pair $(X,D)$ obtained by a blow-up from a toric log Calabi--Yau pair. 
Comparing the cones of relevant curves associated to $(X,D)$ and its degeneration $(\widetilde{X}, \widetilde{D})$ discussed above, we see that $Q(\widetilde{X},\widetilde{D})$ is contained in the monoid generated by the union of $Q(X,D)$, the fiber classes $\pm F_i$'s and classes of exceptional curves $\pm E_i^j$'s. Moreover as there are no relations between the fiber classes $F_i's$ and the classes in $Q(X,D)$, we have a well defined morphism of monoids $Q(\widetilde{X},\widetilde{D}) \to Q(X,D)$ given by setting $ \pm F_i=0$. Hence, by setting all the classes $ \pm F_i=0$ in the wall structure $T_0\foD^{1}_{(\widetilde X,\widetilde D)}$, we obtain a consistent wall structure defined over the localization of $Q(X,D)$ at classes of exceptional curves  (see Definition \ref{Def: M}). We call this wall structure the \emph{heart of the canonical wall structure associated to $(X,D)$} and denote it by $\foD^{\heart}_{(X,D)}$ -- see Definition \ref{def: heart}. 

The advantage of passing to the heart of the canonical wall structure is that it is supported on $\RR^n$ rather than $B$ which carries affine singularities. Particularly, keeping track of broken lines in $\foD^{\heart}_{(X,D)}$ is more convenient. 
Our main result shows that the broken lines on the heart of the canonical wall structure $\foD^{\heart}_{(X,D)}$ define the correct theta functions generating the mirror to $(X,D)$:
\begin{citedthm}[(= Theorem \ref{thm: heart})]
The ring of theta functions defined by broken lines in $\foD_{(X,D)}^{\heart}$ is isomorphic to the coordinate ring of the mirror to $(X,D)$.
\end{citedthm}
As a particular application using the heart of the canonical wall structure we compute explicit equations for mirror families to a log Calabi--Yau pair $(X,D)$ in dimension three, obtained as the blow-up of $\PP^3$ along a disjoint union of hypersurfaces. In the case when we consider the blow-up along more than one hypersurface, writing an explicit equation for the mirror is significantly challenging, since the walls formed using the tropicalizations of the hypersurfaces intersect and at each such intersection there are new walls formed. We show that even in the simplest case, when the center of blow-up is a union of two disjoint lines, we have infinitely many new walls. Nonetheless, we observe that the product of the wall crossing functions on these walls converge and we obtain concrete equations for the mirror -- see \S\ref{Sec: examples}.


We prove that the equation for the mirror, in the situation when one considers the blow-up of a toric variety along a single hypersurface, agrees with previous results of Abouzaid--Auroux--Katzarkov computed from the symplectic point of view \cite[Theorem~1.5]{AAK}:

\begin{citedthm}[(= Theorem \ref{Thm AAK agrees AG})]
Let $X$ be the blow-up of a toric variety along a hypersurface $H$ of its toric boundary and $D$ be the strict transform of the toric boundary divisor. Let $E$ be the class of an exceptional fiber over $H$. Then, the restriction of the mirror $Y\to \mathrm{Spec}\mathbf{k}[Q(X,D)]$ to the locus $\CC^* = \mathrm{Spec}\CC[t^{\pm E}] \subset \mathrm{Spec}\mathbf{k}[Q(X,D)]$ is isomorphic to the mirror constructed in \cite{AAK}.
\end{citedthm}

We note that in some situations the pairs $(X,D)$ obtained by a blow-up of $\PP^3$ with center a disjoint union of hypersurfaces of degrees $d_1$ and $d_2$ are Fano (for instance when $d_1=d_2=1$, or $d_1=d_2=2$). In these cases, the sum of the theta functions we compute, which generate the mirror family, agree with the Landau--Ginzburg superpotential as computed by Coates--Corti--Galking--Kasprzyk \cite{coates2016quantum} (see Remark \ref{Rem CC}). This verifies that the mirror families we compute in these situations are the ones expected from the point of view of Landau--Ginzburg mirror symmetry.


\subsection{Acknowledgements} I thank Dan Abramovich, Pierrick Bousseau, Tom Coates, Mark Gross, and Bernd Siebert for many useful discussions. I am particularly grateful to Tom Coates, who provided the magma code to carry the wall crossing computations in dimension three. During the preparation of this paper, I received funding
from the European Research Council (ERC) under the European Union’s Horizon 2020
research and innovation programme (grant agreement No. 682603), from Fondation
Math\'ematique Jacques Hadamard and from IST Austria.

\textbf{Conventions.} For any variety $X$, we denote by $N_1(X)$ the abelian group generated by projective irreducible curves in $X$ modulo numerical equivalence. Moreover, we denote by $\mathrm{NE}(X) \subset N_1(X) \otimes_{\ZZ} \RR$ the Mori cone, which is the cone generated by effective curves. We use the notation $\langle \rho_1, \ldots ,\rho_n \rangle$
for a cone in $\RR^n$ whose set of ray generators is $\{ \rho_1, \ldots ,\rho_n \}$. 

\section{Mirrors to log Calabi--Yau pairs: the toric case}
\label{Sec: mirrors to toric}
In this section we review the construction of the mirror to  log Calabi--Yau pairs in the context of the Gross--Siebert program \cite{GSCanScat,GHS}, by restricting attention to \emph{toric log Calabi--Yau pairs} $(X_{\Sigma}$,$D_{\Sigma})$ given by an $n$-dimensional toric variety $X_{\Sigma}$ associated to a complete toric fan $\Sigma \subset M_{\RR}$, and the toric boundary divisor $D_{\Sigma}$, that is, the anti-canonical divisor formed by the union of divisors that are invariant under the torus action. To construct the mirror to such a pair we need the following data:
\begin{itemize}
    \item The \emph{tropicalization} of $(X_{\Sigma}$,$D_{\Sigma})$: this is given by the pair $(\RR^n,\Sigma)$,
where $\Sigma$ is naturally viewed as a \emph{polyhedral subdecomposition} of $\RR^n$.
    \item The monoid $Q$ of integral points of the Mori cone $\mathrm{NE}(X_{\Sigma})$, and a convex piecewise linear (PL) function $\overline{\varphi}: \RR^n \rightarrow Q^{\gp}_\RR$, that is, a function whose restriction to each maximal cone of $\Sigma$ is a linear function. Such a function is uniquely determined, up to a linear function, by specifying its \emph{kinks} along codimension one cells of $\Sigma$. For any codimension one cell $\rho$ of $\Sigma$ the kink of $\overline{\varphi}$ which we denote by $\kappa_{\rho}$, up to a choice of sign, is given by the change of slopes of the restriction of $\overline{\varphi}$ to the maximal cells adjacent to $\rho$ -- see \cite[Def.\ 1.6, Prop.\ 1.9]{GHS}. There is a canonical choice for the kinks of $\overline{\varphi}$, which we use in what follows, given along each codimension one cell $\rho$ by the corresponding curve class in $X_{\Sigma}$. In general, by the assumption of convexity of $\overline{\varphi}$ we ensure the kinks are elements of $Q$, rather than $Q^{\gp }$ -- see \cite[Definition~1.10]{GHS}
\end{itemize}
\begin{example}
Let $X_{\Sigma}$ be the complex projective plane $\PP^2$. The Mori cone in this case is given by $Q = \NN = \langle [L] \rangle$. The three rays in $\Sigma$ of the toric fan correspond to lines in $\PP^2$, for which we denote the associated curve class by $[L]$. Let $\overline{\varphi}$ be the PL function defined by 

\begin{equation}
\label{Eq: PL vanishing at positive octant P2}
\overline{\varphi}(x,y) = \begin{cases} 
      0 & \mathrm{on} ~ \langle (1,0),(0,1) \rangle \\
      -y [L] & \mathrm{on} ~ \langle (1,0),(-1,-1) \rangle  \\ 
      -x [L] & \mathrm{on} ~ \langle (0,1),(-1,-1) \rangle 
   \end{cases}
\end{equation}
The PL function $\overline{\varphi}$ has kinks $[L]$ along each of the rays of $\Sigma$. We note that specifying the kinks along each ray, determines uniquely $\overline{\varphi}$ only up to a linear function, as we can always add a multiple of a linear function and the kinks of the resulting PL function will still be the same. However, note that in addition to specifying the kinks, if we ask the PL function to vanish at a given maximal cell, then the choice is unique. We illustrate the three PL functions with kinks $L$, and which vanish on a maximal cell in Figure \ref{Fig: PL}. 
\begin{figure}
\resizebox{.9\linewidth}{!}{\input{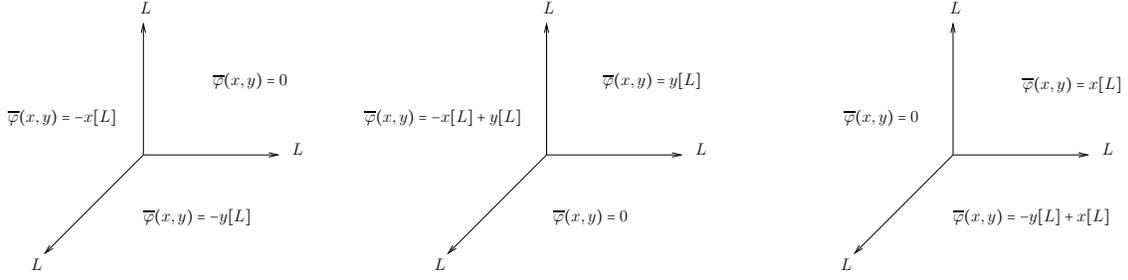}}
\caption{The possible $Q$-valued PL functions on the fan $\Sigma$ of $\PP^2$ with kinks $L$, and which vanish along a maximal cone.}
\label{Fig: PL}
\end{figure}
\end{example}

In the remaining part of this section, by applying the general recipe developed in \cite{GSCanScat} to toric varieties, we explain how to construct the mirror family to a toric log Calabi--Yau pair $(X_{\Sigma},D_{\Sigma})$ as an affine toric variety. In this situation the mirror arises as a family with total space a toric variety, whose momentum map image is given by the polytope formed by the upper convex hull of the graph of $\overline{\varphi}$. More precisely, we define the monoid $P$ of integral points lying above the graph of $\overline{\varphi}$ by
\[ P   := \{(m, \overline{\varphi}(m) +q) ~ | ~ m ~\in M, q \in Q \} \subset M \oplus Q^{\gp} \]
The natural inclusion $Q \hookrightarrow P$ gives rise to a family
\[  \Spec \mathbf{k}[P]  \longrightarrow  \Spec \mathbf{k}[Q],  \]
which is declared to be the mirror family to $(X_{\Sigma}$,$D_{\Sigma})$. Here, the ring $\mathbf{k}[P]$ is called \emph{the ring of theta functions} \cite{GHS}. Indeed, for any integral point $m \in M$, we have a regular function, referred to as a theta function, 
\[\vartheta_m =  z^{(m , \overline \varphi(m) )} \in \mathbf{k}[P]\] on $\Spec \mathbf{k}[P]$. Moreover, the set of theta functions $\{ \vartheta_m \}_{m \in M}$ form a basis for $\mathbf{k}[P] $ as a $\mathbf{k}[Q]$ module. In the situation when the toric variety $X_{\Sigma}$ is smooth, a generating set for  $\mathbf{k}[P]$ as a $\mathbf{k}[Q]$ algebra is given by particular theta functions $\{\vartheta_{m_i}\}_{i\in I}$, where the set of vectors $\{ m_i ~ | ~ i \in I\}$ correspond to the set of primitive generators of rays of the fan $\Sigma$. To write these functions, we fix a general point $p \in M_{\RR}$ contained in the interior of a maximal cell of $\Sigma$, and define $\overline{\varphi}$ to be the PL function which vanishes in the maximal cell containing $p$. Then, we set
\begin{equation}
    \label{Eq: vartheta}
    \vartheta_{m_i}(p) := z^{(m_i, \overline{\varphi}(m_i) )} = z^{m_i} t^{\overline{\varphi}(m_i)} \in \mathbf{k}[P] = \mathbf{k}[M \oplus Q^{\gp}].
\end{equation}
Here we denote for the element $(m_i, \overline{\varphi}(m_i) ) \in P$, the corresponding element in the monid algebra by $z^{(m_i, \overline{\varphi}(m_i) )} \in \mathbf{k}[P]$. Note that we have a natural splitting $P = M \oplus Q^{\gp}$ since the point $p$ is chosen in the interior of a maximal cell $\Sigma$. To distinguish between the elements of the monoid algebras associated to $M$ and $Q^{\gp}$, following the notational convention of \cite{AG} for $m \in M$ we denote the corresponding element in the monoid algebra by $z^m \in \mathbf{k}[M]$, and for $q \in Q^{\gp}$ the corresponding element in the monoid algebra is $t^q \in \mathbf{k}[Q^{\gp}]$.

\begin{example}
\label{Ew: mirror to P2}
For $X_{\Sigma} = \PP^2$, recall we have $Q:=\NN = \langle L \rangle$, where $[L]$ is the class of a line in $\PP^2$. We let $p\in \Sigma$ be a point in the positive octant. Then the Q-valued PL function $\overline{\varphi}$ vanishing at $p$ is defined in \eqref{Eq: PL vanishing at positive octant P2}. For the generators $(1,0)$ and $(0,1)$ of the monoid $M = \ZZ^2$, we denote the corresponding elements in the monoid algebra $\mathbf{k}[M]$ by
\[  z^{(1,0)}=x, ~ \mathrm{and} ~ z^{(0,1)} =y \,. \]
Then, by \eqref{Eq: vartheta}, the theta functions generating the coordinate ring for the mirror to $(X_{\Sigma},D_{\Sigma})$ in this case are given by
\begin{align}
\label{Eq: thetas for P2}
    \vartheta_{(1,0)} & = z^{(1,0)} t^{\overline{\varphi}(1,0)} = x \,, \\
    \nonumber
    \vartheta_{(0,1)} & = z^{(0,1)} t^{\overline{\varphi}(0,1)} =  y \,,  \\
\nonumber
 \vartheta_{(-1,-1)} & = z^{(-1,-1)} t^{\overline{\varphi}(-1,-1)} = x^{-1}y^{-1}  t^{[L]} \,.
\end{align}
It follows from \eqref{Eq: thetas for P2} that the mirror to $(X_{\Sigma},D_{\Sigma})$ is 
\[  \mathrm{Spec}\mathbf{k}[\langle L \rangle ][\vartheta_{(1,0)} ,  \vartheta_{(0,1)} ,  \vartheta_{(-1,-1)}] / (  \vartheta_{(1,0)} \vartheta_{(0,1)} \vartheta_{(-1,-1)}  = t^{L} ) \]
\end{example}

Before proceeding, we give another example of the mirror to a toric log Calabi--Yau pair in dimension three.
\begin{example}
\label{Ew: mirror to P3}
Let $X_{\Sigma} = \PP^3$. For the generators $(1,0,0)$, $(0,1,0)$ and $(0,0,1)$ of the monoid $M = \ZZ^3$, we denote the corresponding elements in the monoid algebra $\mathbf{k}[M]$ by
\[  z^{(1,0,0)}=x, ~ z^{(0,1,0)}=y, ~ \mathrm{and} ~ z^{(0,0,1)} =z\]
We fix a point $p\in \Sigma$ in the interior of the positive octant, and a PL function $\overline{\varphi}$ vanishing at $p$ defined by
\begin{equation}
\label{Eq: PL vanishing at positive octant}
\overline{\varphi}(x,y,z) = \begin{cases} 
      0 & \mathrm{on} ~ ~ \langle (1,0,0),(0,1,0),(0,0,1) \rangle  \\
      -x [L] & \mathrm{on} ~ \langle (0,1,0),(0,0,1),(-1,-1,-1) \rangle \\ 
      -y [L] & \mathrm{on} ~ \langle (1,0,0),(0,0,1),(-1,-1,-1) \rangle \\ 
      -z [L] & \mathrm{on} ~ \langle (1,0,0),(0,1,0),(-1,-1,-1) \rangle \\ 
   \end{cases}
\end{equation}
Then, using \eqref{Eq: vartheta}, we write the theta functions generating the coordinate ring for the mirror to $(X_{\Sigma},D_{\Sigma})$ which in this case are: 
\begin{align}
\label{Eq: thetas for P3}
    \vartheta_{(1,0,0)} & = z^{(1,0,0)} t^{\overline{\varphi}(1,0,0)} = x \,, \\
    \nonumber
    \vartheta_{(0,1,0)} & = z^{(0,1,0)} t^{\overline{\varphi}(0,1,0)} =  y \,, \\
\nonumber
 \vartheta_{(0,0,1)} & = z^{(0,0,1)} t^{\overline{\varphi}(0,0,1)} =  z \,, \\
\nonumber
 \vartheta_{(-1,-1,-1)} & = z^{(-1,-1,-1)} t^{\overline{\varphi}(-1,-1,-1)} = x^{-1}y^{-1}z^{-1}  t^{[L]} \,.
\end{align}
It follows from \eqref{Eq: thetas for P3} that the mirror to $(X_{\Sigma},D_{\Sigma})$ is 
\[  \mathrm{Spec}\mathbf{k}[\langle L \rangle ][\vartheta_{(1,0,0)} ,  \vartheta_{(0,1,0)} , \vartheta_{(0,0,1)}, \vartheta_{(-1,-1,-1)}] / (  \vartheta_{(1,0,0)} \vartheta_{(0,1,0)}\vartheta_{(0,0,1)}\vartheta_{(-1,-1,-1)}  = t^{L} ) \]
\end{example}

\section{Mirrors to log Calabi--Yau pairs: the general case}
\label{sec:canonical scattering}
To construct mirrors to log Calabi--Yau pairs which are not toric, we need a generalization of the notion of a momentum polytope image of the mirror to a toric log Calabi--Yau pair. This is provided by the \emph{canonical wall structure}, or the \emph{canonical scattering diagram}  \cite{GHK,GHS,GSCanScat}. Before describing the canonical wall structure, we first review the general definition of wall structures\footnote{Compare with the most general set-up of \cite{GHS}, we are making some simplifying assumptions which will always be satisfied for the examples considered in this paper: $B$ is taken to be a manifold rather than a general pseudomanifold, and we assume that $\Delta$ is contained in a union of codimension two cells of $\P$.}. 

\subsection{Data for wall-structures}
\label{sec_data}
To define a wall structure we need 
to fix the following data:
\begin{itemize}
    \item $(B,\P)$: {an integral affine manifold with singularities} $B$, along with a polyhedral decomposition $\P$, such that the \emph{discriminant locus} $\Delta$ of the affine structure is contained in a union of codimension two cells of $\P$. In what follows we refer to cells of $\P \subset B$ which are of dimensions $0$, $1$ and $n$ as
\emph{vertices}, \emph{edges} and \emph{maximal cells}. The set of $k$-cells are denoted by $\P^{[k]}$ and we write $\P^{\max}:=\P^{[n]}$ for the set of maximal cells.
We allow $B$ to be a manifold with boundary $\partial B$, that is required to be a union of codimension one cells of $\P$.
Cells of
$\P$ contained in $\partial B$ are called a \emph{boundary cell}, and
cells of $\P$ not contained in $\partial B$ are called \emph{interior}.
We denote by $\mathring{\P}\subseteq \P$ the set of interior cells of $\P$. We denote by $\Lambda$ the sheaf of integral tangent vectors on $B \setminus \Delta$, and for every cell $\sigma$ of $\cP$, we denote by $\Lambda_\sigma$ the space of integral tangent vectors to $\sigma$.
    \item A \emph{toric monoid} $Q$. Recall that a toric monoid $Q$ is a finitely generated, integral, saturated monoid which in addition satisfies that $Q^{\gp}$ is torsion-free. We denote by $Q_{\RR} \subseteq Q_{\RR}^{\gp}$ the corresponding cone, that is, $Q = Q^{\gp} \cap Q_{\RR}$. We denote $I_0:= Q \setminus Q^\star$ the maximal monoid ideal of $Q$, where $Q^\star$ is the set of invertible elements. We also fix a monoid ideal $I$ of $Q$ with radical $I_0$.

\item A \emph{multi-valued piecewise linear} (MVPL) function  $\varphi$ on $B \setminus \Delta $ with values in $Q^{\gp}_{\RR}$:
We define a \emph{multi-valued
piecewise linear} (MVPL) function $\varphi$ on $B\setminus\Delta$ with values in $Q^{\gp}_{\RR}$ as in \cite[Def.\ 1.4]{GHS}. 
On the open star $\mathrm{Star}(\rho)$ of each codimension one cell $\rho\in\mathring\P$,
we have a piecewise
linear function $\varphi_{\rho}$, well-defined up to
linear functions.
Such a MVPL function is determined by specifying its \emph{kinks} $\kappa_{\rho}
\in Q^{\gp}$ for each codimension one cone $\rho\in \mathring{\P}$ defined as follows (see \cite[Def.\ 1.6, Prop.\ 1.9]{GHS}): Let $\rho\in \mathring{\P}$ be a codimension one cone and let $\sigma,\sigma'$ be the two maximal cells
containing $\rho$, and let $\varphi_{\rho}$ be a piecewise linear
function on $\mathrm{Star}(\rho)\subset
B\setminus\Delta$. An affine chart at $x\in\mathrm{Int}\rho$ thus
provides an identification $\Lambda_\sigma=
\Lambda_{\sigma'}=:\Lambda_x$. Let $\delta:\Lambda_x\to\ZZ$ be the quotient by
$\Lambda_\rho \subseteq\Lambda_x$. Fix signs by requiring that $\delta$
is non-negative on tangent vectors pointing from $\rho$ into
$\sigma'$. Let $n,n'\in \check\Lambda_x\otimes Q^\gp$ be the slopes of
$\varphi_{\rho}|_\sigma$, $\varphi_{\rho}|_{\sigma'}$, respectively. Then
$(n'-n)(\Lambda_\rho)=0$ and hence there exists $\kappa_{\rho}\in Q^\gp$
with
\begin{equation}
\label{Eqn: kink}
n' -n =\delta \cdot\kappa_{\rho}.
\end{equation}
We refer to $\kappa_{\rho}$ as the \emph{kink} of $\varphi_{\rho}$ along $\rho$.
Thus, if $\varphi$ is an MVPL function, it has a well-defined kink
$\kappa_{\rho}$ for each such $\rho$, and these kinks determine $\varphi$. We also assume that the MVPL function $\varphi$ is \emph{strictly convex} in the sense that $\kappa_\rho
\in I_0$ for all $\rho$.
\item An \emph{order zero function} $f_{\underline{\rho}} \in (\mathbf{k}[Q]/I_0) [\Lambda_\rho]$ for each codimension one cell $\rho$ of $\mathring{\P}$. 
\end{itemize}

The choice of the MVPL function $\varphi$ gives rise to a local system $\shP$ fitting into the exact sequence
    \begin{equation}
\label{eq:Q exact}
0\rightarrow \ul{Q}^{\gp}\rightarrow \shP\rightarrow \Lambda\rightarrow 0 .
\end{equation}
Here $\ul{Q}^{\gp}$ is the constant sheaf with stalk $Q^{\gp}$. The sheaf $\shP$ contains via \cite[Def.\ 1.16]{GHS}, a subsheaf $\shP^+\subseteq\shP$.

\begin{itemize}
    \item For a generic point $x\in B$ in the interior of a maximal cell $\P^{\mathrm{max}} \in \P$, the stalk of $\shP^+$ is $\shP_x^+=\Lambda_x\times Q$, whereas the stalk of $\shP$ is $\shP_x=\Lambda_x\times Q^{\gp}$.
 \item For a point $x$ lies in the interior of a codimension one cell
$\rho$ which is not a boundary cell, 
\begin{equation}
\nonumber
\label{Eq: Pplus description}
\shP^+_x=\big\{\big(m,(d\varphi_{\rho}|_{\sigma})(m)+q\big)\,\big |\,\rho\subseteq\sigma
\in\scrP^{\max},\, m\in T_x\sigma\cap\Lambda_x,\,q\in Q\big\}.
\end{equation}
Here $T_x\sigma$ denotes the tangent wedge to $\sigma$ at $x$.
\end{itemize}
For an element $m\in \shP_x$, we write $\bar m\in\Lambda_x$ for its image
under the projection of \eqref{eq:Q exact}.


\subsection{Wall-structures}
Now we are ready to define a wall structure. 
\begin{definition}
\label{Def: walls}
Fix an integral affine manifold with singularities along with a polyhedral decomposition $(B,\P)$, a toric monoid $Q$, a strictly convex MVPL function $\varphi$, and order zero functions $f_{\underline{\rho}}$ as in \S \ref{sec_data}. A \emph{wall} on $(B,\P)$ is a codimension one rational polyhedron
$\fod\not\subseteq\partial B$ contained in some maximal cone $\sigma$ of $\P$, 
along with an element
\begin{equation}
\label{Eq: canonical fncs}
    f_{\fod}=\sum_{m\in\shP^+_x, \bar m\in \Lambda_{\fod}} c_m z^m 
\in \mathbf{k}[\shP^+_x]/I_x,
\end{equation}
referred to as a \emph{wall crossing function}, where $c_m\in \mathbf{k}$.  
Here $x\in \mathrm{Int}(\fod)$ and
$\Lambda_{\fod}$ is the lattice of integral tangent vectors to
$\fod$. We require that $m \in \shP^+_x$ for all $y \in \fod \setminus \Delta$ when $c_m \neq 0$. We say a wall $\fod$ has \emph{direction} $v \in \Lambda_{\fod}$ if the attached function $f_\fod$, given as in \eqref{Eq: canonical fncs}, satisfies 
$\bar m=-kv$ for some $k\in \NN$ whenever $c_m\not=0$. We call a wall with direction $v$ \emph{incoming} if $\fod=\fod-\RR_{\ge 0} v$. A \emph{wall structure} or a \emph{scattering diagram} on $(B,\P)$ over $Q$ is a finite set $\foD$ of walls on $B$ given as in \eqref{Eq: canonical fncs}, and satisfying the following conditions:
\begin{itemize}
    \item If $\fod \cap \mathrm{Int} \sigma \neq \emptyset$  then $f_{\fod} \equiv 1$ modulo $I_0$, and
    \item For every codimension one cell $\rho$ of $\mathring{\P}$, and every point $x \in \rho$, denote by $f_{\rho,x}$ the product of $f_\fod$ over all the walls $\fod$ containing $x$ and contained in $\rho$. Then, we have $f_{\rho,x} \equiv f_{\underline{\rho}}$ modulo $I_0$.
\end{itemize}
\end{definition}

If $\foD = \cup (\fod,f_{\fod})$ is a wall structure, we define the \emph{support} and the \emph{singular locus} in $\foD$ respectively by
\begin{eqnarray}
\nonumber
\mathrm{Supp}(\foD) & := & \bigcup_{\fod } \fod, \\
\nonumber
\mathrm{Sing}(\foD) & := & \Delta \cup \bigcup_{\fod } \partial\fod
\cup \bigcup_{\fod,\fod'} (\fod\cap\fod') \,,
\end{eqnarray}
where the last union is over all pairs of walls $\fod,\fod'$ with
$\fod\cap\fod'$ codimension at least two. In particular,
$\mathrm{Sing}(\foD)$ is a codimension at least two subset of $B$.

\subsection{The canonical wall structure}
\label{sec: wall structures}
To define the canonical wall structure associated to a log Calabi--Yau pair $(X,D)$ we first describe the tropicalization $(B,\P)$ of $(X,D)$, then the monoid $Q(X,D)$ associated to $(X,D)$ along with a $Q^{\gp}_{(X,D)}$-valued PL function on $(B,\P)$. 
\subsubsection{The tropicalization $(B,\P)$ of $(X,D)$}
The tropical space associated to $(X,D)$, or the \emph{tropicalization of $(X,D)$}, is a pair $(B,\P)$ consisting of an integral affine manifold with singularities $B$, along with a polyhedral decomposition $\P$. We describe $(B,\P)$ from the data of the intersection numbers of irreducible components of $D$. For this, first consider $\Div(X)$, which denotes the group of divisors on $X$, and $\Div_D(X)\subseteq \Div(X)$, the subgroup of divisors supported
on $D$. Moreover, we set
\[\Div_D(X)_{\RR}=\Div_D(X)\otimes_{\ZZ}\RR.\] 
Let $D=\bigcup_{i=1}^m D_i$ be the decomposition of $D$ into
irreducible components, and write $\{D_i^*\}$ for the dual basis of
$\Div_D(X)_{\RR}^*$. We assume throughout that for any index subset
$I\subseteq \{1,\ldots,m\}$, if
non-empty, $\bigcap_{i\in I} D_i$ is connected. Define the polyhedral decomposition $\P$ to be the collection of cones
\begin{equation}
    \label{Eq: polyhedral decomposition}
    \P:=\left
\{\sum_{i\in I} \RR_{\ge 0} D_i^*\,|\, \hbox{$I\subseteq \{1,\ldots,m\}$
such that $\bigcap_{i\in I}D_i\not=\emptyset$}\right\}.
\end{equation}
Then we set
\[
B:=\bigcup_{\tau\in\P}\tau \subseteq \Div_D(X)^*_{\RR}.
\]

Generally, we view the tropicalization $(B,\P)$ of a log Calabi--Yau pair $(X,D)$ as a topological manifold described as above, together with the data of an affine structure with singularities -- see \cite{AG}[\S2.1.1]. 


\begin{example}
\label{Ex: trop X}
Let $X$ be a del Pezzo surface of degree $8$. Thus, $X$ is isomorphic to the blowup $X \to \PP^2$ in a single point, which we assume to lie in the interior of a component of the toric boundary divisor $D_{\PP^2} \subset \PP^2$. We set $D$ to be the strict transform of $D_{\PP^2}$. Then, $(X,D)$ defines a log Calabi--Yau pair. In this case, $D$ has $3$ irreducible components with self-intersection numbers given by the tuple $(1,1,0)$. The associated tropical space $B$ has three maximal two dimensional cones, whose set of rays are given respectively by $\{\rho_1,\rho_2\}$, $\{\rho_2,\rho_3\}$ and $\{\rho_3,\rho_1\}$, where $\rho_1$ has direction $(1,0)$, $\rho_2$ has direction $(0,1)$, and $\rho_3$ has direction $(-1,-1)$. We denote the cone with rays $\{\rho_i,\rho_j\}$ by $C_{i,j}$.  For this consider an open cover of $\RR^2 \setminus \{ 0 \}$ given by the union of the three subsets 
\begin{align*}
    U_1 & = C_{1,2} \cup  C_{1,3} \setminus \{ \rho_2,\rho_3 \}, \\
    U_2 & = C_{1,3} \cup  C_{2,3} \setminus \{ \rho_1,\rho_3 \}, \\
    U_3 & = C_{1,2} \cup  C_{2,3} \setminus \{ \rho_1,\rho_3 \}, \\
\end{align*}
and define the charts for the affine structure by setting $\Psi_1 : U_1 \hookrightarrow \mathbb{R}^2$ $\Psi_2 : U_2 \hookrightarrow \mathbb{R}^2$ to be restrictions of the identity map on $U_1$ and $U_2$ respectively. We then define $\Psi_3:U_3 \to \mathbb{R}^2$ by
\[ 
\Psi_3(x,y) = \begin{cases} 
      (x,y) & \mathrm{on} ~ C_{1,2}  \setminus \{ \rho_1 \} \\
      (x,y-x) & \mathrm{on} ~ C_{1,2}  \setminus \{ \rho_3 \}
   \end{cases}
\]
as illustrated in Figure \ref{Fig: affine}. Note that the matrix for the change of coordinate transformation in this case is conjugate to
\begin{equation*}
\begin{pmatrix}
1 & 1  \\
0 & 1  
\end{pmatrix}
\end{equation*}
which represents the standard \emph{focus-focus singularity} -- see for instance \cite{leung2010almost} for further discussion on such singularities in dimension two and the affine monodromy. This endows $B$ with an integral affine structure with a singularity at the origin.
\begin{figure}[h!]
\resizebox{.9\linewidth}{!}{\input{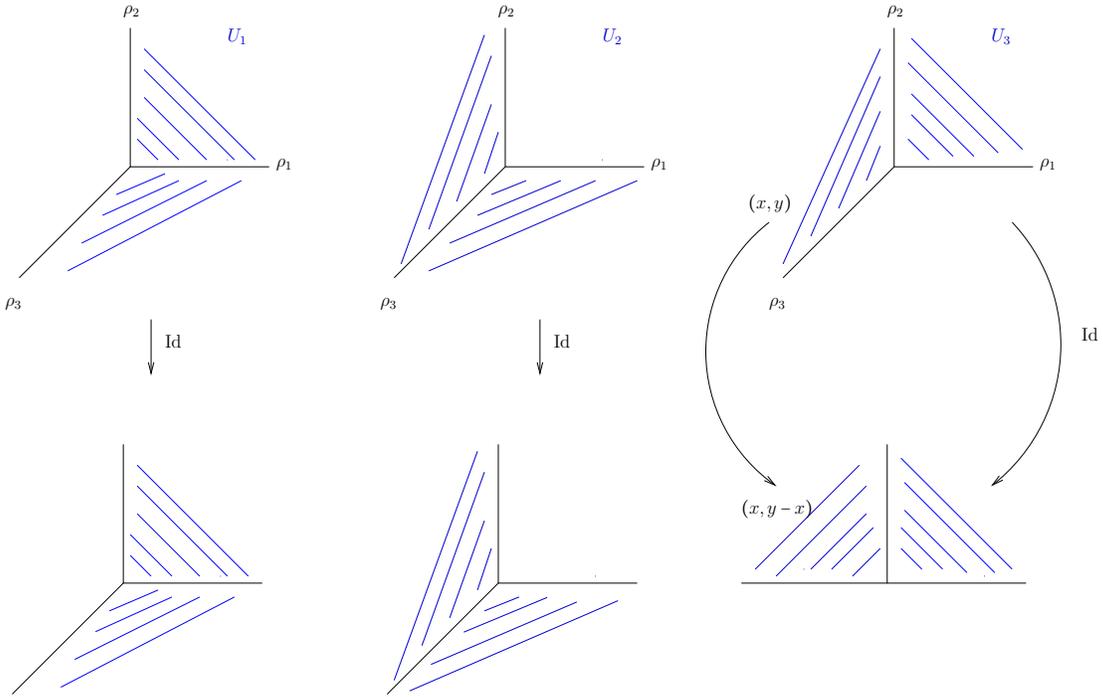}}
\caption{The three charts defining the integral affine structure on $B \setminus \{  0 \}$}
\label{Fig: affine}
\end{figure}
\end{example}

The next ingredient we need to define the canonical wall structure associated to a log Calabi--Yau pair $(X,D)$ is the toric monoid, which we denote by $Q(X,D)$ and refer to as the relevant monoid, and the data of a MVPL function with values in $Q^{\gp}_{\RR}(X,D)$, which is specified by its kinks in $Q(X,D)$. 
\subsubsection{The relevant monoid $Q(X,D)$}
To define $Q(X,D)$, we first need the description of $\AA^1$-curves and boundary curves on $(X,D)$.
\begin{definition}
\label{A1curve}
An \emph{$\AA^1$-curve} on a log Calabi--Yau pair $(X,D)$ is the image of a genus zero stable map to $X$, such that the intersection of $C$ with $D$ is a single point.
\end{definition}
Observe that by the description of the tropicalization of $(X,D)$, it automatically follows that in the situation  $(X,D)$ is a blow-up of a toric log Calabi--Yau pair as in \eqref{Eq: blow up}, any codimension one stratum on $B$ corresponds to a rational curve in $X$ contained in $D$. More generally, for any log Calabi--Yau pair $(X,D)$ since by definition $D$ has simple normal crossing singularities, such a strata corresponds to a smooth curve.  
\begin{example}
Let $X$ be the blow-up of a non-toric point in the interior of the toric boundary divisor in $\PP^2$. Then, an exceptional curve with class $E$ as well as a curve with class $L-E$, where $L$ is the class of a general line in $X$ as illustrated in Figure \ref{Fig:0} are examples of $\AA^1$-curves.
\end{example}

To describe the relevant monoid, in addition to $\AA^1$-curves, we also consider \emph{boundary curves} in $(X,D)$, which are curves contained in $D$.


\begin{definition}
\label{QXD}
Let $(X,D)$ be a log Calabi--Yau pair.
The \emph{relevant cone of curves} $\mathcal{C}(X,D)$ is the cone in $N_1(X) \otimes \RR$ generated by the union of all $\AA^1$-curves and boundary curves.
The \emph{relevant monoid $Q(X,D)$ associated to $(X,D)$} is the monoid of integral points in $\mathcal{C}(X,D)$:
\begin{equation}
    \label{eq: monoid for XD}
Q(X,D) := \langle  [ C ] ~ | ~ C \mathrm{~is ~  an ~} \AA^1-\mathrm{curve \mathrm{~ or ~  a  ~ boundary ~ curve ~}} \rangle_{\ZZ}.
\end{equation}
Here we us the notation $[C]$ to denote the class of a curve $C$.
\end{definition}
Before proceeding, we show that in the two dimensional situation, the relevant cone of curves agrees with the Mori cone ``generically''. To describe the notion of \emph{genericity} for a log Calabi--Yau pair in dimension two, we need the following definition, which can be found in \cite[Definition~1.5]{gross2015moduli}:
\begin{definition}
\label{Def: generic}
Let $(X,D)$ be a log Calabi--Yau pair and assume that $X$ is of dimension two. Denote by $D^{\perp} \subset \mathrm{Pic}(X)$ be the sublattice of the Picard group of $X$, defined by  
\[ D^{\perp} := \{ \alpha \in \mathrm{Pic}(X) ~ | ~ \alpha \cdot [D_i] = 0 ~ \mathrm{for ~ all ~ i} \}.\]
There is a natural \emph{period map} 
\begin{align}
     \label{Eq: period}
    \phi_X :  D^{\perp} & \longrightarrow \mathrm{Pic}^0(D) \cong \CC^* \\
    \nonumber
    \mathcal{L} & \longmapsto  \mathcal{L}|_{D}
\end{align}
defined by restricting a line bundle on $X$ to $D$.
\end{definition}
A key result in \cite{gross2015moduli} shows that the deformation space of a log Calabi--Yau pair $(X,D)$, where $X$ is of dimension two, is locally isomorphic to $\mathrm{Hom}(D^{\perp},\CC^*)$. Therefore, it makes sense to define a generic log Calabi--Yau pair as follows.
\begin{definition}
\label{Def: generic log CY}
A log Calabi--Yau pair $(X,D)$, where $X$ is of dimension two, is called generic if $\phi_X(\alpha)  \neq 1$
for all $\alpha \in D^{\perp} $ where $\phi_X$ is the period map defined in \eqref{Eq: period}.
\end{definition}
\begin{remark}
The definition of genericity we provide here is slightly different than in \cite[Definition~1.4]{gross2015moduli}, which is equivalently stated in \cite[Corollary~3.5]{gross2015moduli} as the condition $\phi_X(\alpha)  \neq 1$ for all $\alpha \in D^{\perp} $ which have self-intersection $-2$. Here we require this condition for all $\alpha$ regardless of the self-intersection. It is also worthwhile mentioning that, frequently the term \emph{generic} is used for the complement of finitely many objects, while here we have not finite but countably many objects, as the condition $\phi_X(\alpha) = 1$ defines in most cases an infinite union of hypersurfaces in $\mathrm{Hom}(D^{\perp}, \CC^*)$. If we would consider only $\alpha$ with self intersection $-2$ we would still have countably many hypersurfaces, rather than finite. So, we inherit the abuse of the term ``generic'' here from \cite{gross2015moduli}.
\end{remark}

\begin{theorem}
\label{Theorem relevant Mori}
Let $(X,D)$ be a generic log Calabi--Yau pair, where $X$ is of dimension two. Then, the 
relevant cone of curves $\mathcal{C}(X,D)$ in Definition \ref{QXD} is isomorphic to the Mori cone $\mathrm{NE}(X)$. 
\end{theorem}

\begin{proof}
By definition as $\mathcal{C}(X,D)$ is generated by the union of boundary curves, together with $\AA^1$-curves, any element in $\mathcal{C}(X,D)$ is an element of the Mori cone $\mathrm{NE}(X,D)$. For the converse, first note that the statement can be easily verified when $X$ is $\PP^2$ or a Hirzebruch surface. In a more general situation, given any irreducible effective curve $C$ in a generic log Calabi--Yau pair $(X,D)$ of dimension two, we will show that it lies in $\mathcal{C}(X,D)$, by analysing the following three possible cases:
\begin{itemize}
 \item $C \cdot K_X = 0$: Consider the line bundle $\mathcal{O}_X(C)$, which has a canonical section that vanishes exactly on $C$. So, as  $\mathcal{O}_X(C)$ is trivial away from $C$, the image of it under the period map \eqref{Eq: period} is trivial, that is, $\phi_X(\mathcal{O}_X(C))=1$. Thus, in this case $X$ is not generic.
    \item $C \cdot K_X < 0$: By Mori's cone theorem, the part of the Mori cone with $K_X < 0$ is generated by extremal rays. Either these would be $-1$ curves or $X$ is a  Hirzebruch surface or $\PP^2$ -- see \cite[\S~5.4]{debarre2016introduction}. Hence, excluding the latter cases, the result follows since any $-1$ curve, by the adjunction formula is a rational curve intersecting $D$ at a single point, hence is in $\mathcal{C}(X,D)$.
      \item $C \cdot K_X > 0$: In this case we have $C \cdot D < 0$, and hence $C$ is enforced to be a boundary curve. 
\end{itemize}
Hence, the result follows.
\end{proof}

\begin{remark}
\label{Rem: cant go higher}
The natural generalization of Theorem \ref{Theorem relevant Mori} to higher dimensions does not hold. For instance, consider the log Calabi--Yau pair $(X,D)$, where $X$ is obtained by a blowing-up $4$ disjoint lines, each contained in one of the toric boundary components of $\PP^3$, and let $D$ be the strict transform of the toric boundary. Then, since there always exists at least one line passing through the $4$ lines that we blow up, there will always be at least one effective curve in the interior of $X$ obtained as the strict transform of such a line, which does not correspond to an element of $\mathcal{C}(X,D)$. In conclusion, generally the mirror family constructed in \cite{GSCanScat} is a base change from a family over the smaller base given by the formal completion of $\mathrm{Spec}\mathbf{k}[Q(X,D)]$ at the maximal ideal $Q(X,D) \setminus \{0\}$.
\end{remark}

We proceed with the description of the final ingredient needed to define the canonical wall structure, that is, a multi-valued piecewise-linear (MVPL) function.

\subsubsection{The MVPL function $\varphi : B\setminus \Delta \to Q(X,D)^{\gp}_\RR$}
\label{subsec:varphi}
Now we are ready to describe the MVPL function $\varphi : B\setminus \Delta \to Q(X,D)^{\gp}_\RR$, by specifying its kinks on codimension one cells of $(B,\P)$. There is a canonical choice of these kinks, defined as follows: for a codimension one cell $\rho\in\mathring{\P}$,
set the kink of $\varphi$ to be 
\begin{equation}
    \label{Eq: kink}
    \kappa_{\rho}=[D_{\rho}],
\end{equation}
the class of the boundary curve corresponding to $\rho$ (see \cite[\S2.1.2]{AG} for further details). Note that fixing the kinks, uniquely determines $\varphi$ up to a linear function. Before proceeding, we provide an example of a multi-valued PL function.
\begin{example}
Let $(B,\P)$ be the tropical space as in Example \ref{Ex: trop X}, associated to the log Calabi--Yau pair $(X,D)$, where $X$ is a non-toric blow up of $\PP^2$. To define a MVPL function on $B \setminus \{ 0 \}$ it suffices to define a piecewise linear function on the neighbourhoods given by the open stars of each of the three rays $\rho_i$, for $1 \leq i \leq 3$. These functions, up to linear functions are defined by specifying the kinks $[L]\in Q$, given by the class of a general line in $X$, along each ray $\rho_2$ and $\rho_3$, and $[L-E] \in Q$ along $\rho_1$ where $[E]$ stands for the class of the exceptional fiber. 
Note that in their domains of intersections these PL functions may take different values, as $\rho$ is indeed "multi-valued".
\end{example}

\subsubsection{The canonical wall-structure}
Now, we are ready to review definition of the canonical wall structure associated to $(X,D)$ following \cite[\S~$2.4$]{AG}, or \cite{GSCanScat}. We let $(X,D)$ be a log Calabi--Yau pair with tropicalization $(B,\P)$, and $Q(X,D)$ the relevant monoid associated to $(X,D)$ defined as in \eqref{eq: monoid for XD}. We also fix a MVPL function with kinks defined canonically as in \eqref{Eq: kink}. Finally, for every codimension one cell $\rho$ of $\P$, we consider the order zero functions $f_{\underline{\rho}}=1$. Note that $Q(X,D)^\star=\{0\}$, so $I_0=Q(X,D)\setminus \{0\}$, and so $\mathbf{k}[Q(X,D)]/I_0 \simeq \mathbf{k}$.

The \emph{canonical wall structure} associated to $(X,D)$ is a wall structure on $(B,\P)$ over $Q(X,D)$, where for each wall $\fod$, the attached wall crossing function is concretely given by \eqref{Eq: canonical wall}. Note that, for every ideal $I \subseteq Q(X,D)$ such that 
$\sqrt{I}=I_0$, considering the wall crossing functions modulo $I$, the canonical wall structure is a wall structure in the sense of Definition \ref{Def: walls}, that is, with finitely many walls. If we do not work modulo such ideal $I$, the canonical wall structure might contain infinitely many walls.
\begin{example}
The canonical wall structure associated to the blow-up of $\PP^2$ at a non-toric point is illustrated on the left hand side of Figure \ref{Fig:3}.
\end{example}

\subsection{Pulling singularities out from the canonical wall structure}
 \label{Pulling singularities}
In this section we review how to ``pull out'' the discriminant locus of the canonical wall structure associated to a log Calabi--Yau pair $(X,D)$, obtained from a toric log Calabi--Yau pair $(X_{\Sigma},D_{\Sigma})$ by a blow-up as in \eqref{Eq: blow up}. 
More precisely, we fix distinct rays 
$\rho_1,\dots,\rho_s$
of the fan $\Sigma$ of $X_\Sigma$, and a disjoint union of smooth hypersurfaces $H=H_1\cup\cdots\cup H_s$ in $D_\Sigma$, such that $H_i \subset D_{\rho_i}$ for all $1\leq i \leq s$, where $D_{\rho_i}$ is the toric divisor of $X_\Sigma$ corresponding to the ray $\rho_i$. Then, we take for $X$ the blow-up of $X_\Sigma$ along $H$, and for $D$ the strict transform of $D_\Sigma$. We assume further that the toric variety $X_\Sigma$ is smooth and projective, and that no cone of $\Sigma$ contains two rays $\rho_i$ and $\rho_j$ with $i \neq j$. These conditions are always satisfied after refining enough the fan $\Sigma$. 

We further write 
\begin{equation}
    \label{Eq: hypersurfaces H}
H_i=\bigcup_{j=1}^{s_i} H_{ij} \,,
\end{equation}
for the decomposition of $H_i$ into its connected components. The main result of \cite{AG} provides a combinatorial algorithm to construct the canonical wall structure $\foD_{(X,D)}$ associated to $(X,D)$ from a \emph{toric wall structure} $\foD_{(X_{\Sigma},H)}$ in $\RR^n$, obtained from the data of $X_{\Sigma}$ and $H$. We show that this toric wall structure, in rough terms, encodes all the data of the canonical wall structure with its singularities are pulled out. We first provide a precise description of the toric wall structure $\foD_{(X_{\Sigma},H)}$, and then explain how to obtain the canonical wall structure associated to $(X,D)$ from it in the remaining part of this section. For details, we refer to \cite{AG}.

\subsubsection{The toric wall structure}
Let
\begin{equation}
\nonumber
    P = M\oplus \bigoplus_{i=1}^s\NN^{s_i},
\end{equation}
where $M$ is the cocharacter lattice associated with $X_{\Sigma}$, so that the fan $\Sigma$ is contained in the real vector space $M_\RR:=M \otimes \RR$, and let $P^{\times}$ be the group of units of $P$. Consider the ideal $\fom_P=P\setminus P^{\times}$, and denote by $\widehat{\mathbf{k}[P]}$ the completion of $\mathbf{k}[P]$ with respect to
$\fom_P$. We denote the generators of $\NN^{s_i}$ by $e_{i1},\ldots,e_{is_i}$, and set 
\begin{equation}
    \label{Eq: tij}
    t_{ij}:=z^{e_{ij}}\in \widehat{\mathbf{k}[P]}.
\end{equation}

\begin{definition}
A \emph{wall structure on $M_\RR$} is a wall structure as in Definition \ref{Def: walls}, where $B=M_\RR$ with the integral affine structure induced by $M \subset M_\RR$, and $\P$ is the trivial polyhedral decomposition with the single cell $M_\RR$. In particular, there is no MVPL function $\varphi$ or order zero functions in the description of a wall structure in $M_\RR$. Note also that $M_\RR$ is an integral affine manifold without singularities and so the discriminant locus $\Delta$ is empty.
\end{definition}

We review below the definition of the wall structure $\foD_{(X_\Sigma,H)}$ on $M_\RR$ over $P$.
We first describe the initial wall structure $\foD_{(X_{\Sigma},H),\mathrm{in}}$ whose walls are codimension one subsets of $M_\RR$ called \emph{widgets}. We review the description of widgets below. For details, see \cite[\S$5.1.2$]{AG}. 

For every $1\leq i\leq s$, we denote by $m_i \in M$ the primitive generator of the ray $\rho_i$ of $\sigma$. The corresponding \emph{widget} $\foD_{m_i}$ is the wall-structure on $M_\RR$ over $P$ defined as follows:
\begin{equation}\label{eq_initial}
    \foD_{m_i} := \bigcup_{\rho} \left(\rho, \prod_{j=1}^{s_i}(1+t_{ij}z^{m_i})^{D_{\rho}\cdot H_{ij}}\right) \,,
\end{equation}
where the union is over the codimension one cones $\rho$ of $\Sigma$ containing the ray $\rho_i=\RR_{\geq 0} m_i$, and $D_{\rho}\cdot H_{ij}$ is the intersection number in $D_{\rho_i}$ between the hypersurface $H_{ij}$ and the toric curve $D_\rho$ corresponding to $\rho$.

Now, the initial wall structure
$\foD_{(X_{\Sigma},H),\mathrm{in}}$
is defined as the union of the widgets $\foD_{m_i}$:
\[  \foD_{(X_{\Sigma},H),\mathrm{in}} := \bigcup_{i=1}^s \foD_{m_i} \,. \]
We describe the \emph{consistent} wall structure $\foD_{(X_{\Sigma},H)}$ in $M_\RR$ obtained from  $\foD_{(X_{\Sigma},H),\mathrm{in}}$ in a moment, after reviewing the notion of consistency for a wall structure.

\subsubsection{The notion of consistency}
In this section we shortly review the notion of path-ordered products and consistency for a wall structure $\foD$ in $M_\RR$, after setting up a couple of necessary notations. 

Let $\gamma:[0,1]\rightarrow M_\RR$ be a piecewise smooth path
whose image is disjoint from $\Sing(\foD)$. Further, assume that
$\gamma$ is transversal to $\mathrm{Supp}(\foD)$, in the sense that
if $\gamma(t_0)\in\fod\in\foD$, then there is an $\epsilon>0$ such
that $\gamma((t_0-\epsilon,t_0))$ lies on one side of $\fod$
and $\gamma((t_0,t_0+\epsilon))$ lies on the other. Assuming that $\gamma(t_0)\in \fod$, we associate a \emph{wall-crossing}
homomorphism $\theta_{\gamma,\mathfrak{d}}$ as follows.
Let $n_\fod$ be a generator of $\Lambda_\fod^\perp
\subseteq\check\Lambda_x=\Hom(\Lambda_x,\ZZ)$ for some $x\in\mathrm{Int}\fod$, with
$n_\fod$ positive on $\gamma((t_0-\epsilon,t_0))$.
Then define 
\begin{equation}
\label{Eqn: theta_fop}
\theta_{\gamma,\mathfrak{d}}: \quad
z^m\longmapsto f_\fod^{\langle n_\fod,\ol m\rangle} z^m.
\end{equation}
We may now define the \emph{path-ordered product}
\begin{equation}
    \label{Eq: path ordered product}
    \nonumber
    \theta_{\gamma,\foD}:=\theta_{\gamma,\fod_s}\circ\cdots\circ
\theta_{\gamma,\fod_1},
\end{equation}
where $\fod_1,\ldots,\fod_s$ is a complete list of walls traversed
by $\gamma$, in the order traversed.
\begin{definition}
A \emph{joint} is a codimension two polyhedral
subset of $M_\RR$ contained in  $\Sing(\foD)$, and such that for $x\in\mathrm{Int}(j)$, 
the set of walls $\{\fod\in\foD\,|\,x\in \fod\}$ is independent of
$x$. Further, a joint must be a maximal subset with this property. A wall structure on $M_\RR$ is said to be \emph{consistent} if all path ordered products along any sufficiently small loop around a joint is identity. 
\end{definition}
In \cite[Theorem 5.6]{AG}, we prove the higher dimensional analogue of the Kontsevich--Soibelman Lemma \cite{KS}:
\newcommand {\Scatter} {\operatorname{{\mathsf{S}}}}
\begin{theorem}
There is a consistent wall structure
$\foD_{(X_{\Sigma},H)}$ on $M_\RR$ over $P$ containing $\foD_{(X_{\Sigma},H),\mathrm{in}}$ such that $\foD_{(X_{\Sigma},H)}\setminus
\foD_{(X_{\Sigma},H),\mathrm{in}}$ consists only of non-incoming walls. Further, this wall structure is unique up to equivalence.
\end{theorem}
\subsubsection{From $\foD_{(X_{\Sigma},H)}$ to $\foD_{(X,D)}$}
\label{Sec: Degeneration}
To compare $\foD_{(X_{\Sigma},H)}$ with the canonical wall structure associated to $(X,D)$, first note that there is a natural piecewise-linear isomorphism
\[
\Upsilon:(M_{\RR},\Sigma)\rightarrow (B,\P).
\]
The existence of such a piecewise linear isomorphism follows from the definition of the tropicalization of
$(X,D)$ and we refer to \cite[\S6]{AG} for details. For every $1\leq i \leq s$ and $1\leq j \leq s_i$, let $E_i^j$ denote an exceptional
curve of the blow-up over the hypersurface $H_{ij}$. There is a natural splitting
\[N_1(X)=N_1(X_{\Sigma})\oplus\bigoplus_{ij} \ZZ E_i^j \,,\] in which
$N_1(X_{\Sigma})$ is identified with the set of curve classes
in $N_1(X)$ with intersection number zero with all exceptional
divisors. We will define $\Upsilon(\fod,f_{\fod})$, to describe a wall of $\foD_{(X,D)}$ on $B$. This definition
depends on whether $(\fod,f_{\fod})$ is incoming or not. 

If the wall is incoming, then by construction of $\foD_{(X_{\Sigma},H)}$
it is of the form
$(\fod,(1+t_{ij}z^{m_i})^{w_{ij}})$ for some positive integer
$w_{ij}$, see \eqref{eq_initial}. As $m_i$ is tangent to the cone of $\Sigma$ containing
$\fod$ and $\Upsilon$ is piecewise linear with respect to $\Sigma$,
$\Upsilon_*(m_i)$ makes sense as a tangent vector to $B$.
We then define
\[
\Upsilon(\fod,(1+t_{ij}z^{m_i})^{w_{ij}})
=(\Upsilon(\fod),(1+t^{E_i^j}z^{-\Upsilon_*(m_i)})^{w_{ij}}).
\]
If the wall $\fod$ is not incoming, then still the attached function $ f_{\fod}$ is necessarily a power-series in the expression $\prod_{i,j}(t_{ij}z^{m_i})^{a_{ij}}$, for some positive integers $a_{ij}$. We assume after refining the walls of $\foD_{(X_{\Sigma},H)}$ that $\fod\subseteq\sigma\in\Sigma$. Then
the data $\mathbf{A}=\{a_{ij}\}$ determines a curve class 
$\bar\beta_{\mathbf{A}}\in N_1(X_{\Sigma})$
as follows.
Up to a linear function, there exists a unique piecewise linear function
\[ \psi \colon M_\RR \rightarrow N_1(X_\Sigma) \otimes \RR\] with kink
along a codimension one cone $\rho$ being the class of the corresponding one-dimensional stratum $D_\rho \subset X$.
Then, we define
\[ \bar\beta_{\mathbf{A}} :=  \psi(-\sum_{i,j}a_{ij}m_i) 
+ \sum_{i,j}\psi(a_{ij}m_i)\,.
\]

Under the inclusion $N_1(X_{\Sigma}) \hookrightarrow N_1(X)$ 
given by the above mentioned splitting, we may view
$\bar\beta_{\mathbf{A},\sigma}$ as a curve class in $N_1(X)$, which we also denote by $\bar\beta_{\mathbf{A},\sigma}$. We then obtain a curve class
\[
\beta_{\mathbf{A}}=\bar\beta_{\mathbf{A}} - \sum_{ij} a_{ij} E_i^j.
\]
Further, as $m_{\out}:=-\sum_{ij} a_{ij}m_i$ is tangent to the
cone of $\Sigma$ containing $\fod$, as before $\Upsilon_*(m_{\out})$ makes sense
as a tangent vector to $B$. We may thus define the wall
\begin{equation}
\label{Eq: wall upsilon canonical}
\Upsilon(\fod,f_{\fod})=(\Upsilon(\fod),
f_{\fod}(t^{\beta_{\mathbf{A}}}z^{-\Upsilon_*(m_{\out})})).
\end{equation}
We then define 
\[
\Upsilon(\foD_{(X_{\Sigma},H)}) :=
\{\Upsilon(\fod,f_{\fod})\,|\, (\fod,f_{\fod})\in \foD_{(X_{\Sigma},H)}\}.
\]
A key result in \cite[Theorem 6.1]{AG}, then states:
\begin{theorem}
\label{Thm: HDTV}
$\Upsilon(\foD_{(X_{\Sigma},H)})$ is equivalent to $\foD_{(X,D)}$.
\end{theorem}
 Here, two wall structures are equivalent if they induce the same
wall-crossing automorphisms. In the remaining part of this section, we summarise the proof.

To prove Theorem \ref{Thm: HDTV}, we first consider a degeneration $(\widetilde X, \widetilde D)$ over $\AA^1=\mathrm{Spec}\mathbf{k}[t]$ obtained from a blow-up of the degeneration to the normal cone of $X_{\Sigma}$, with general fiber $(X,D)$, and central fiber given by
\begin{equation}
\label{Eq: central fiber of degeneration}
X_{\Sigma} \cup \bigcup_{i=1}^s \mathrm{Bl}_{H_i}(\PP(\mathcal{N}_{D_{\rho_i}|X_{\Sigma}} \oplus \mathcal{O}_{D_{\rho_i}}))   
\end{equation}
We then describe the canonical wall structure associated to the total space $(\widetilde X, \widetilde D)$. It is a wall structure on the tropicalization $(\widetilde B, \widetilde \P)$ of
of $(\widetilde X,\widetilde D)$ over the relevant monoid $Q(\widetilde X,\widetilde D)$.
The tropicalization of the degeneration map $\widetilde X \rightarrow \AA^1$ defines a projection map
\[ \widetilde p: \widetilde B \to \RR_{\geq 0}\,,  \]
and so we obtain a wall structure $\foD^1_{(\widetilde X, \widetilde D)}$ on $\widetilde B_1:= \widetilde p^{-1}(1)$ over $Q(\widetilde X, \widetilde D)$ by restriction to $p^{-1}(1)$, see \cite[\S 3.3]{AG} for details. The singularities of the integral affine manifold $\widetilde B_1$ are away from the origin: from $B=p^{-1}(0)$ to $\widetilde B_1=p^{-1}(1)$, the singularities are pushed away from the origin, see Figure \ref{Fig:3}. We use the notation $\P_1$ to denote the restriction of the polyhedral decomposition $\widetilde{\P}$ on $\widetilde{B}$ to $\widetilde{B}_1$. 
Localizing to the origin $0\in \widetilde B_1$ we obtain a wall structure 
\begin{equation}
T_0\foD^{1}_{(\widetilde X,\widetilde D)}:=\{(T_0\fod, f_{\fod})\,|\,
(\fod,f_{\fod})\in \foD^{1}_{(\widetilde X,\widetilde D)}, \quad 0\in \fod\} 
\end{equation}
on the tangent space  $T_0\widetilde B_1$ of $\widetilde B_1$ at the origin (see \cite[\S5]{AG}), and where $T_0 \fod$ is the tangent space at the origin of the wall $\fod$.
More precisely, as the origin is a smooth point of the integral affine structure on $\widetilde B_1$, we have a natural identification $M_{\RR} \to T_0\widetilde B_1$ such that the fan $\Sigma$ is $M_\RR$ is the restriction of $\P_1$ to $T_0\widetilde B_1$. Moreover, the MVPL function for $(\widetilde X, \widetilde D)$ restricts to the $PL$ function $\varphi_0$ for the toric pair $(X_\Sigma, D_\Sigma)$, that is, with kink $[D_\rho]$ across a codimension one cone $\rho$ of $\Sigma$, where $[D_\rho]$ is the corresponding toric curve class in $X_\Sigma$. Then, $T_0\foD^{1}_{(\widetilde X,\widetilde D)}$ is a wall structure on $(M_\RR,\Sigma)$ over $Q(\widetilde X,\widetilde D)$ as in Definition \ref{Def: walls}, where one uses the toric PL function $\varphi_0$, and where the order zero functions are $f_{\underline{\rho}}=1$ for every codimension one cone $\rho$ of $\Sigma$.

The main technical result of \cite{AG}, \cite[Theorem 6.2]{AG}, is a comparison between the wall structure $T_0\foD^{1}_{(\widetilde X,\widetilde D)}$ on $(M_\RR,\Sigma)$ over $Q(\widetilde X,\widetilde D)$ 
with the wall structure $\foD_{(X_{\Sigma},H)}$ on $M_\RR$ over $P$. There is a map 
\begin{align}\label{eq_nu}
\nu \colon \mathbf{k}[P]& \longrightarrow \mathbf{k}[\cP_0^+]\\ \nonumber t_{ij}z^{m_i} & \longmapsto z^{(m_i,\varphi_0(m_i)+F_i-E_i^j)}\,,
\end{align}
where $F_i$ denotes the class of a general $\PP^1$ fiber of $\mathrm{Bl}_{H_i}(\PP(\mathcal{N}_{D_{\rho_i}|X_{\Sigma}} \oplus \mathcal{O}_{D_{\rho_i}}))$.
Then, the walls of $T_0\foD^{1}_{(\widetilde X,\widetilde D)}$ are obtained from the walls $(\fod, f_{\fod})$ of $\foD_{(X_{\Sigma},H)}$ by applying $\nu$ to $f_\fod$:
\begin{equation} \label{eq_nu_comp}
T_0\foD^{1}_{(\widetilde X,\widetilde D)} \simeq \nu(\foD_{(X_{\Sigma},H)})\,.
\end{equation}


As a second step we consider the asymptotic wall structure $\foD_{(\widetilde X,\widetilde D)}^{1,\mathrm{as}}$, defined by 
\begin{equation}
\label{Eq: height one}
 \foD_{(\widetilde X,\widetilde D)}^{1,\mathrm{as}} :=   \{(\fod\cap \widetilde{B}_0, f_{\fod})\,|\, \hbox{$(\fod,f_{\fod})\in
\foD_{(\widetilde X,\widetilde D)}$ with $\dim\fod\cap\widetilde{B}_0=n-1$}\}.
\end{equation} 
We show that $\foD_{(\widetilde X,\widetilde D)}^{1,\mathrm{as}}$ is equivalent to $\iota(\foD_{(X,D)})$ -- here we view the canonical wall structure $\foD_{(X,D)}$ as a wall structure that is embedded into $\foD_{(\widetilde X, \widetilde D)}$, which we denote by $\iota(\foD_{(X,D)})$. Finally, we show that there is a natural piecewise-linear isomorphism $\mu: M_{\RR} \lra \widetilde B_0= \widetilde{p}^{-1}(0) \cong B$
which induces the equivalence of wall structures $\mu(T_0\foD^{1}_{(\widetilde X,\widetilde D)})$ and $\foD^{1,\mathrm{as}}_{(\widetilde X, \widetilde D)}$, and hence $\iota(\foD_{(X,D)})$:
\begin{equation}\label{eq_mu}
\mu(T_0\foD^{1}_{(\widetilde X,\widetilde D)}) \simeq \iota(\foD_{(X,D)}) \,.
\end{equation}
The map $\Upsilon$ is then given by the composition $\Upsilon= \mu \circ \nu$.



\begin{example}
\label{Ex0}
Let $X$ be the blow-up of a non-toric point in the interior of the boundary divisor $D_{\Sigma} \subset \PP^2$, and $D$ be the strict transform of $D_{\Sigma}$. The central fiber of the degeneration $\widetilde{X}$ of $X$ is then given as a union of $\PP^2$ and the Hirzebruch surface $F_1= \PP(\mathcal{O} \oplus \mathcal{O}(-1))$, as illustrated in Figure \ref{Fig:0}
\begin{figure}
\resizebox{.7\linewidth}{!}{\input{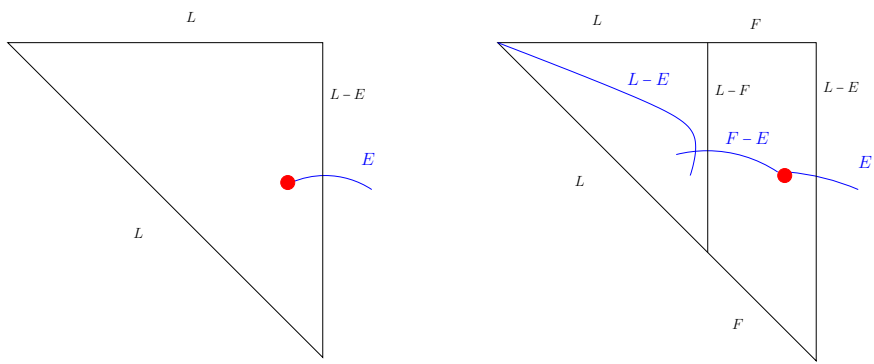}}
\caption{The momentum polytope picture associated to $X$, the blow-up of $\PP^2$ at a non-toric point, on the left and the central fiber of the degeneration of $\widetilde{X}$ of $X$ on the right. The curve classes corresponding to one dimensional cells are drawn in black. The exceptional curve $E$ illustrated on the left contributes to the canonical wall structure of $(X,D)$, while the curves illustrated on the right contribute to the canonical wall structure of the degeneration $(\widetilde X,\widetilde D)$.}
\label{Fig:0}
\end{figure}
A possible choice for the piecewise linear function 
$\psi$ is $\psi(x,y)=0$ on the cone $<(1,0),(0,1)>$,
$\psi(x,y)=-x L$ on the cone $<(0,1),(-1,-1)>$, 
$\psi(x,y)=-y L$ on the cone $<(0,1),(-1,-1)>$, 
where $L$ is the class of a line in $\PP^2$, generating 
$NE(\PP^2)$. We have $m_1=(1,0)$ and $m_2=(0,1)$.
So, when applying $\Upsilon$ to the function 
$1+t_1x=1+t_1 z^{m_1}$ attached to a non incoming wall, we obtain $1+xt^{\bar{\beta}_{\mathbf{A}}}$, where 
$\beta_{\mathbf{A}}=\bar{\beta}_{\mathbf{A}} -E$, and 
\[ \bar{\beta}_{\mathbf{A}}=\psi(-m_1)+\psi(m_1)=L+0=L\,.\]
In other words, the function 
$1+t_1 x$ attached to a non incoming wall becomes 
$1+xt^{L-E}$.  We illustrate the canonical wall structure associated to $(X,D)$, and the height one slice of the canonical wall structure associated to the degeneration $(\widetilde{X},\widetilde{D})$ in Figure \ref{Fig:3}.
\begin{figure}
\resizebox{.9\linewidth}{!}{\input{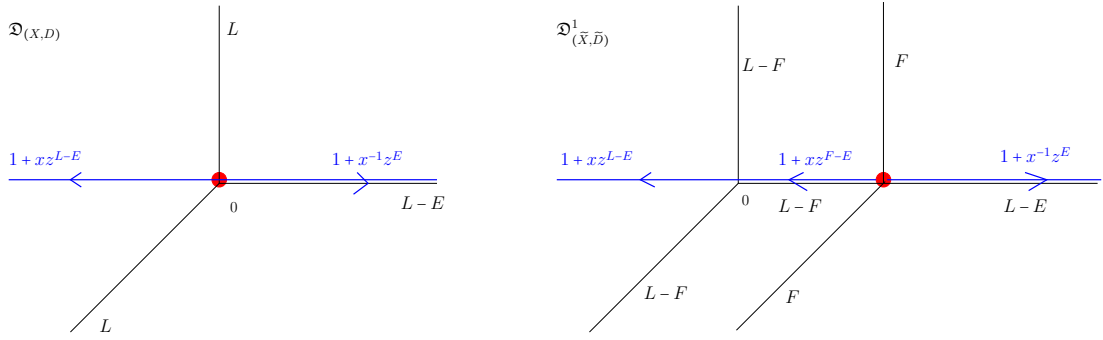}}
\caption{The canonical wall structure $\foD_{(X,D)}$ associated to the blow up of $\PP^2$ at a single non-toric point on the left, the height one slice of the canonical wall structure $\foD^1_{(\widetilde{X},\widetilde{D})}$ associated to the degeneration $(\widetilde{X},\widetilde{D})$ in the right. 
Here walls and attached functions are in blue and the one cells of the polyhedral decomposition as well as kinks of the PL functions on them are in black.}
\label{Fig:3}
\end{figure}
\end{example}

\subsection{Theta functions defined by broken lines}
\label{sec: theta functions defined by broken lines}
As shown in \cite{GHS}, the mirror to a log Calabi--Yau pair $(X,D)$ -- or rather the mirror to the complement $X \setminus D$ is a family $\mathrm{Spec}\mathcal{R}_{X^{\vee}}$ over the formal completion of $\mathrm{Spec}\mathbf{k}[Q(X,D)]$ at the maximl ideal $Q(X,D) \setminus \{0\}$,
where $\mathcal{R}_{X^{\vee}}$ denotes the \emph{ring of theta functions} associated to $(X,D)$. The generators of this ring, referred to as theta functions, are defined combinatorially via broken lines in the canonical wall structure $\foD_{(X,D)}$. Below we first review the definition of broken lines. In what follows, we show that the generators of the ring of theta functions for the mirror to a log Calabi--Yau pair $(X,D)$ as in \eqref{Eq: blow up} can actually be obtained by studying broken lines in the ``heart'' of the canonical wall structure. This will allow us to compute the theta functions concretely, and to obtain concrete equations for the mirrors.

To define broken lines on $(B,\P)$ we need some notations for the local rings defined by considering the monoids over the graphs of the MVPL function $\varphi$ on $B$ discussed in \S\ref{sec: wall structures}. As the restriction of such a function to maximal cells is linear, the monoid above the graph of such a cell takes a rather simple form. Indeed, for a maximal cell $\sigma\in\P^{\max}$ with $x\in\mathrm{Int}(\sigma)$, we set
\begin{equation}
\label{Eqn: R_fou}
R_\sigma:=\mathbf{k}[\shP^+_x]/I_x =  (\mathbf{k}[Q]/I)[\Lambda_\sigma]
\end{equation}
where we have a natural splitting $\shP^+_x = \Lambda_x \times Q$. On the other hand, for a codimension one cell $\rho$ of $\P$
not contained in the boundary of $B$, we set $R_{\rho}:= \mathbf{k}[\shP^+_x]/I_x$, for $x\in\mathrm{Int}(\rho)$. However, in this case the description of $\shP^+_x$ requires some more care, and involves the kinks of  -- see \cite[Equation ~(2.13)]{AG} for details. Now, we are ready to define broken lines.
\begin{definition}
\label{Def: broken line}
Let $Q$ be a toric monoid and $\foD$ a wall structure on $(B,\P)$ over $Q$. A
\emph{broken line} in $\foD$ is a piecewise linear continuous directed
path 
\begin{equation}
    \label{Eq:  broken line}
    \beta \colon (-\infty,0] \lra B \setminus \Sing(\foD)
\end{equation}
with $\beta(0)\not\in\mathrm{Supp}(\foD)$ and
whose image consists of finitely many line segments $L_1, L_2, \ldots , L_N$, such that 
$\dim L_i\cap\fod =0$ for any wall $\fod\in\foD$,
and each $L_i$ is compact except $L_1$. Further, we require that each $L_i\subseteq\sigma_i$
for some $\sigma_i\in\P^{\max}$. To each such segment we assign a monomial 
\[     m_i :=  \alpha_iz^{(v_i ,q_i)} \in \mathbf{k} [\Lambda_{L_i} \oplus Q^{\gp}].\]
Here $\Lambda_{L_i}$, as usual, denotes the group of integral tangent vectors
to $L_i$ and is hence a rank one free abelian group.
Each $v_i$ is non-zero and tangent to $L_i$, with $\beta'(t)=-v_i$
for $t\in (-\infty,0]$ mapping to $L_i$.
We require $\alpha_1=1$ and set $m_1 = z^{(v_1,0)}$. We refer to $v_1$ 
as the \emph{asymptotic direction} of the broken line, and to $\beta(0)$ as the \emph{end-point}. Given $L_i$ and its attached monomial $m_i$, we determine
$L_{i+1}$ and $m_{i+1}$ as follows. Let $L_i$ be the image under
$\beta$ of an interval $[t_{i-1},t_i]\subset (-\infty,0]$.
Let $I=[t_i-\epsilon,t_i+\epsilon]$ be an interval with $\epsilon$
chosen sufficiently small so that $\beta([t_i-\epsilon,t_i))$
and $\beta((t_i,t_i+\epsilon])$ are disjoint from $\mathrm{Supp}(\foD)$.
There are two cases:
\begin{itemize}
    \item $\beta(t_i)\in\mathrm{Int}(\sigma_i)$ for $\sigma\in\P^{\max}$.
Then we obtain a wall-crossing automorphism $\theta_{\beta|_I,\foD}:
R_{\sigma_i}\rightarrow R_{\sigma_i}$, and $m_i$ may be viewed as an element
of $R_{\sigma_i}$ via the inclusion $\Lambda_{L_i}\subseteq\Lambda_{\sigma_i}$. 
We expand $\theta_{\beta|_I,\foD}(m_i)$ as
a sum of monomials with distinct exponents, and require that
$m_{i+1}$ be one of the terms in this sum.
\item $\beta(t_i)\in\mathrm{Int}(\rho)$ for $\rho\in\P$ a codimension one
cell. If $y=\beta(t_i-\epsilon)$, $y'=\beta(t_i+\epsilon)$, $x=\beta(t_i)$,
we may view $(v_i,q_i)\in\shP^+_y$. By parallel transport to
$x$ along $\beta$, we may view $(v_i,q_i)\in \shP_x$. In fact,
$(v_i,q_i)\in \shP_x^+$ by the assumption that $\beta'(t_i-\epsilon)=-v_i$
and \cite[Proposition~2.7]{AG}. Thus, we may view
$m_i\in R_{\rho}$, and then $m_{i+1}$ is required to be a term
in $\theta_{\beta|_I,\foD}(m_i)$. A priori, $m_{i+1}\in R_{\rho}$, but
it may be viewed as a monomial in $R_{\sigma_{i+1}}$ by parallel
transport to $y'$. 
\end{itemize}
We call the monomial $a_Nz^{(v_N,q_N)}$, carried by the final segment $L_N$ of a broken line $\beta$ the \emph{final monomial} carried by $\beta$. If  $v_1=\ldots = v_N$ we say $\beta$ is \emph{never-bending}.
\end{definition}

Definition \ref{Def: broken line}, roughly put, says that a broken line $\beta$ with asymptotic direction $v$, starts its life coming from infinity with a monomial $z^{(v,0)}$ and ends at a fixed endpoint in $B$. Each time $\beta$ crosses a wall of $\foD$ it either goes straight, or bends in the direction of the wall. If it goes without bending it only may gets a contribution from the kink of the PL function, otherwise when it bends the monomial $z^{(v,0)}$ gets multiplied with the monomial term in the wall crossing function attached to the wall. 

Now we are ready to define theta functions from broken lines following \cite[\S~3.3]{GHS}. 
\begin{definition}
\label{Def: theta defined by broken lines}
Let $\foD$ be a wall structure on $(B,\P)$ over $Q$. Fix a general point $p$ in the interior of a cell $\sigma \in \P^{\max}$. Let $m \in B$ be an asymptotic direction for $\P$, that is, a direction of an unbounded ray of $\P$. Then, the \emph{theta function defined by broken lines in $\foD$ with asymptotic direction $m$ and end point $p$} is defined by 
 \begin{equation}
     \label{Eq: theta function defined by broken line}
     \vartheta_m(p) := \sum_{\beta} a_N z^{(v_N,q_N)} \in R_{\sigma}
 \end{equation}
where the sum runs over all broken lines $\beta$ with asymptotic direction $m$, and end-point $p$, and $a_N z^{(v_N,q_N)} $ are the corresponding final monomials, as in Definition \ref{Def: broken line}.
\end{definition}
Given a log Calabi--Yau pair $(X,D)$ with tropicalization $(B,\P)$, it is shown in \cite[\S~3.3]{GHS} that the theta functions defined by the broken lines in the canonical wall structure $\foD_{(X,D)}$, with end-point at a general fixed point and asymptotic directions given by asymptotic directions of $\P$, form the generators for the coordinate ring for the mirror to $(X,D)$. This is easy to verify for the case of a toric log Calabi--Yau pair $(X_{\Sigma},D_{\Sigma})$ -- in this situation we view the tropicalization $(\RR^n,\Sigma)$ endowed with the data of a PL function as discussed in \S\ref{Sec: mirrors to toric} a trivial wall structure, where the wall crossing functions on all walls given by codimension one cells of $\Sigma$, are identity.
\begin{example}
The theta functions generating the mirror to the toric log Calabi--Yau pair $(X_{\Sigma},D_{\Sigma})$ for $X_{\Sigma} = \PP^2$, defined by never-bending broken lines are illustrated in Figure \ref{Fig:brokenP2}. Note that they agree with the theta functions in \eqref{Eq: thetas for P2}, defined without using broken lines. 
\begin{figure}
\resizebox{.9\linewidth}{!}{\input{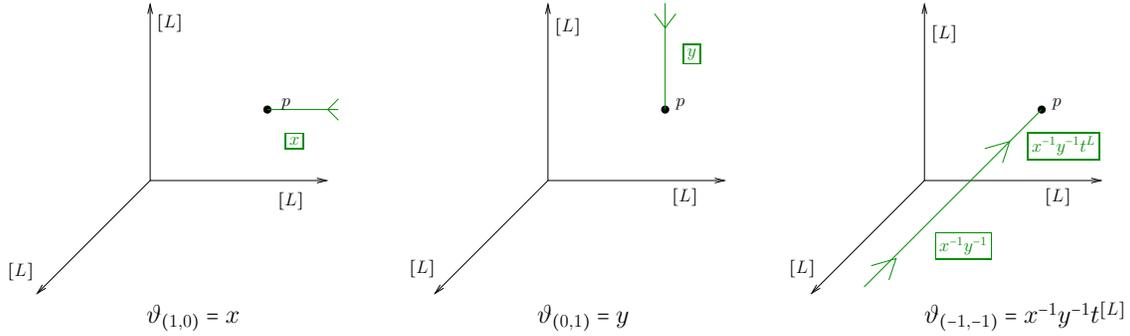}}
\caption{The theta functions generating the coordinate ring for the mirror to $(\PP^2,D_{\Sigma})$ are defined by never-bending broken lines. These broken lines are illustrated in green, and the monomials carried by each of the segments of these broken lines are drawn in boxes also in green. The $[L]$ along each ray is the kink of the PL function.}
\label{Fig:brokenP2}
\end{figure}
\end{example}
Generally, due to the existence of the discriminant locus in the tropicalization $(B,\P)$ of a non-toric log Calabi--Yau pair $(X,D)$, it is challenging keeping track of all broken lines defining theta functions. In the following section, we show that in the situation when $(X,D)$ arises as a blow-up as in \eqref{Eq: blow up}, the generators of the coordinate ring to the mirror of $(X,D)$ are given by broken lines in \emph{the heart of the canonical wall structure associated to $(X,D)$}, and these are easier to keep track of.

\section{The heart of the canonical wall structure}
\label{Sec: the heart of the canonical wall structure}
Let $(X,D)$ be a log Calabi--Yau pair obtained as a blow-up of a toric pair as in \eqref{Eq: blow up} and $(\widetilde{X},\widetilde{D})$ its degeneration described in \S \ref{Sec: Degeneration}. Recall that the corresponding wall structure $T_0\foD^{1}_{(\widetilde X,\widetilde D)}$ in \eqref{Eq T0 intro} is obtained by restricting the canonical wall structure of $(\widetilde X,\widetilde D)$ to height one, and localizing around the origin. In this section we define the \emph{heart of the canonical wall structure associated to $(X,D)$} using $T_0\foD^{1}_{(\widetilde X,\widetilde D)}$. For this, we first fix a monoid defined as follows.

\begin{definition}
\label{Def: M}

Let $(X,D)$ be the blow-up of a toric log Calabi--Yau pair $(X_\Sigma, D_\Sigma)$ as in \eqref{Eq: blow up} along a union of hypersurfaces $H_i \subset D_{\rho_i}$ in the toric boundary where $1\leq i \leq s$, and where $D$ is the strict transform of the toric boundary divisor $D_{\Sigma} \subset X_{\Sigma}$. Denote by $H_i=\bigcup_{j=1}^{s_i} H_{ij}$ the decomposition of $H_i$ into connected components, and $E_i^j$ an exceptional curve over $H_{ij}$. We define \emph{the relevant monoid to $(X,D)$ localized at $E_i^j$}, as the monoid obtained from the relevant monoid associated to $(X,D)$ in \eqref{eq: monoid for XD} by adding the opposite of each exceptional curve $E_i^j$, and denote it by 
  \begin{equation}
    \label{eq: heart monoid}
Q_{E}(X,D) := \langle [D_{\rho}], [ C ], -[ E_i^j ] \, \mathrm{where \,} \, 1\leq i \leq s \, \mathrm{and \,} 1\leq j \leq s_i \rangle_{\ZZ}\,,
\end{equation}
where $[D_{\rho}]$ and $[ C ]$ are as in \eqref{eq: monoid for XD}.
\end{definition}

Note that unlike $Q(X,D)$, the monoid $Q_E(X,D)$ has non-trivial invertible elements: we have \[ Q_E(X,D)^\star=\bigoplus_{i=1}^s \bigoplus_{j=1}^{s_i} \Z [E_i^j]\,.\]

\begin{definition} \label{def_wall_str}
A \emph{wall structure on $(M_\RR ,\Sigma)$ over $Q_{E}(X,D)$} is a wall structure as in Definition \ref{Def: walls}, where $B=M_\RR$ with the integral affine structure induced by $M \subset M_\RR$,  $\P=\Sigma$, the MVPL function is the toric PL function $\varphi_0$, and the order zero functions are given by 
\[f_{\underline{\rho}}=\prod_{j=1}^{s_i} (1+z^{(m_i, \varphi_0(m_i))-E_i^j)})^{D_\rho \cdot H_{ij}}\,,\] when $\rho$ is a codimension one cone of $\Sigma$ containing the ray $\rho_i$, and $f_{\underline{\rho}}=1$ if $\rho$ is a codimension one cone of $\Sigma$ containing none of the rays $\rho_i$.
\end{definition}

\begin{definition}
\label{def: heart}
Let $(X,D)$ be a log Calabi--Yau pair with tropicalization $(B,\P)$, obtained from a toric log Calabi--Yau pair $(X_{\Sigma},D_{\Sigma})$ by a blow-up as in \eqref{Eq: blow up}. The \emph{heart of the canonical wall structure} associated to $(X,D)$, denoted by $\foD_{(X,D)}^{\heart}$, is the wall structure on $(M_\RR,\Sigma)$ over $Q_{E}(X,D)$, obtained from the wall structure $T_0\foD^{1}_{(\widetilde X,\widetilde D)}$ in \eqref{Eq T0 intro} by setting all classes $F_i=0$, where $F_i$ denotes the class of a general $\PP^1$ fiber of $\mathrm{Bl}_{H_i}(\PP(\mathcal{N}_{D_{\rho_i}|X_{\Sigma}} \oplus \mathcal{O}_{D_{\rho_i}}))$, and $H_i$ is as in Definition \ref{Def: M}.
\end{definition}

Note that by the construction of the degeneration $(\widetilde X,\widetilde D)$, elements of the monoid $Q(\widetilde{X},\widetilde{D})$ are contained in the monoid generated by the union of $Q(X,D)$ and the fiber classes $\pm F_i$'s. As there are no relations between the fiber classes $F_i's$ and the classes in $Q(X,D)$, we have indeed a well defined morphism of monoids $Q(\widetilde{X},\widetilde{D}) \to Q(X,D)$ given by setting $ \pm F_i=0$.

Moreover, one can check that $\foD_{(X,D)}^{\heart}$ is indeed a wall structure on $(M_\RR,\Sigma)$ over $Q_E(X,D)$. If $(\fod, f_\fod)$ is a non incoming wall of $T_0\foD^{1}_{(\widetilde X,\widetilde D)}$, then it follows from \eqref{eq_mu} that $(\fod, f_\fod)$ can be viewed as a wall of $\foD_{(X,D)}$, and so the curve classes appearing in $f_\fod$, which are a priori in $Q(\widetilde X,\widetilde D)$, are actually contained in $Q(X,D)$. In particular, setting $F_i=0$ has no effect on the non incoming walls $(\fod, f_\fod)$, and we have $f_\fod \equiv 1$ modulo $I_0=Q_E(X,D) \setminus Q_E(X,D)^\star$. On the other hand,  it follows from the comparison with $\foD_{(X_\Sigma,H)}$ given in \eqref{eq_nu}-\eqref{eq_nu_comp} and from the description of incoming walls of $\foD_{(X_\Sigma,H)}$ in \eqref{eq_initial} that the initial walls of $T_0\foD^{1}_{(\widetilde X,\widetilde D)}$ are 
\[ (\rho, \prod_{j=1}^{s_i}(1+z^{(m_i,\varphi_0(m_i)+F_i-E_i^j})^{D_{\rho}\cdot H_{ij}})\,,\]
for codimension one cones $\rho$ of $\Sigma$ containing a ray $\rho_i$.
Setting $F_i=0$, we obtain that the initial walls of $\foD_{(X,D)}^{\heart}$ are 
$(\rho, f_{\underline{\rho}})$, where $f_{\underline{\rho}}$ is as in Definition \ref{def_wall_str}. Finally, note that $\foD_{(X,D)}^{\heart}$ is a consistent wall-structure because the wall-structure $T_0\foD^{1}_{(\widetilde X,\widetilde D)}$ is consistent.


In the remaining part of this section we prove our main result, Theorem \ref{thm: heart} which shows that the mirror to a log Calabi--Yau pair can be effectively constructed using the heart of the associated canonical wall structure. For this, we first review the analysis we carried in \cite[Lemma 4.20]{AG} to determine the monodromy around each of the standard pieces of the discriminant locus in the tropicalization $(\widetilde B, \widetilde{\P})$ of $(\widetilde X, \widetilde D)$.
We fix some notations to do this. 
\begin{notation}
\label{notations needed}
Let $\rho_i$ be a ray in the toric fan $\Sigma$ with primitive generator $m_i$, corresponding to a divisor $D_{\rho_i} \subset X_{\Sigma}$. We define the toric fan corresponding to $D_{\rho_i}$ by
\begin{equation}
    \label{Eq: sigma rho}
\Sigma(\rho_i)=\{(\sigma+\RR\rho_i)/\RR\rho_i\,|\,\sigma\in\Sigma,
\rho_i\subseteq\sigma\}.
\end{equation}
We use the notation $\ul \rho$ for a codimension one cell of $\Sigma(\rho_i)$ satisfying $\ul{\rho}=(\rho+\RR\rho_i)/\RR\rho_i$. We denote by $\bar\rho_i:=\rho_i\times\RR_{\ge 0}\in \Sigma \times \RR_{\geq 0}$. Note that $\bar\rho_i$ is generated by $(m_i,0)$ and $(0,1)$. Recall that a standard piece of the discriminant locus in $\overline{B}_1$ is located at $(m_i,1)$, and the ray connecting it to the origin it splits the $\bar\rho_i$ into a of two cones: one of them generated by $(m_i,1)$ and $(0,1)$ denoted by $\tilde \rho$ and the other generated by $(m_i,1)$ and $(0,1)$ which we denote by  $\tilde \rho'$. We use the notation $\rho_0$ and $ \rho_\infty$ for the intersections of $\tilde \rho$ and $\tilde \rho_\infty$ with $\widetilde{B}_1$ respectively. Moreover we denote the maximal cells adjacent to $\rho_0$ and $\rho_\infty$ respectively by $\sigma'^{\pm}$ and $\sigma^{\pm}$  as illustrated in Figure \ref{Fig:note}. 
\end{notation}
To describe the monodromy around the singular locus $\Delta \subset \widetilde B_1$, we need the data of a PL function, which is different than the PL function $\varphi$ we fix through \S\ref{sec: wall structures} to describe wall structures (see \cite[Eqn3.14]{AG} for details). We review the description this function in a moment. Denote by $D_{\rho_i}\subset X_{\Sigma}$ the divisor corresponding to a ray $\rho_i$ with direction $m_i$ and let $H_i \subset D_{\rho_i}$ be a hypersurface as in \eqref{Eq: hypersurfaces H}. Denote by $\Sigma(\rho_i)$ the toric fan corresponding to $D_{\rho_i}$ defined as in \eqref{Eq: sigma rho}. Then, there is a piecewise linear function on $\Sigma(\rho_i)$, given by
\begin{equation}
    \label{Eq: varphi-i}
    \varphi_i:M_{\RR}/\RR\rho_i \rightarrow \RR
\end{equation}
corresponding to the divisor $H_i$, defined as follows: if $H_i$
is linearly equivalent to a sum $\sum a_\tau D_{\tau}$ of boundary
divisors, where $\tau$ ranges over rays in $\Sigma(\rho_i)$,
then $\varphi_i(m_{\tau})=a_{\tau}$ for $m_{\tau}$ a primitive generator of $\tau$. 
\begin{proposition}
\label{Prop: monodormy}
The monodromy around a loop $\gamma$ in $\widetilde{B}_1$ around a piece of the discriminant locus on a wall with direction $m_i$ is given by the formula  
 \begin{equation}  
 \label{Eq: afffine monodromy}
T_{\gamma}(m)= m + \kappa_{\ul \rho}^i\cdot \delta(m) \cdot m_i.
\end{equation}
where $\kappa_{\ul \rho}^i$ is the kink of the PL function $\varphi_i$ defined in \eqref{Eq: varphi-i} along a codimension one cone $\ul{\rho}\in\Sigma(\rho_i)$ given by $\ul{\rho}=(\rho+\RR\rho_i)/\RR\rho_i$ for some codimension one cone
$\rho\in\Sigma$ containing $\rho_i$, and $\delta:M/\ZZ m_i\to\ZZ$ is the quotient by
$\Lambda_{\ul\rho}$ (see \cite[Eqn~$3.28$]{AG}).
\end{proposition}
\begin{proof}
See \cite[Lemma~$3.6$]{AG}.
\end{proof}
Using the description of monodromy in \eqref{Eq: afffine monodromy}, we define the parallel transport map as follows. With the notation of Proposition \ref{Prop: monodormy}, it follows from \eqref{Eq: afffine monodromy} that there is a natural parallel transport map on $\widetilde{B}_1$, given by
\begin{align}
   \label{eq:parallel transport}
   \wp:\mathbf{k}[\Lambda][Q(X,D)] & \longrightarrow 
\mathbf{k}[\Lambda][Q(X,D)] \\
  \nonumber
   t^qz^m & \longmapsto t^qz^{m + \langle m,n \rangle m_i} 
\end{align}
where $n \in N = \mathrm{Hom}(M,\ZZ)$ is the normal vector to $\rho_i$ pointing away from $\widetilde{\sigma}^+ \cup \widetilde{\sigma}'^+$. Note that by the definition of lift $\delta: M/\ZZ m_i\to\ZZ$ in Proposition \ref{Prop: monodormy} it naturally lifts to $n: M \to\ZZ $, so that $\delta(m)=n$.
\begin{figure}
\resizebox{.9\linewidth}{!}{\input{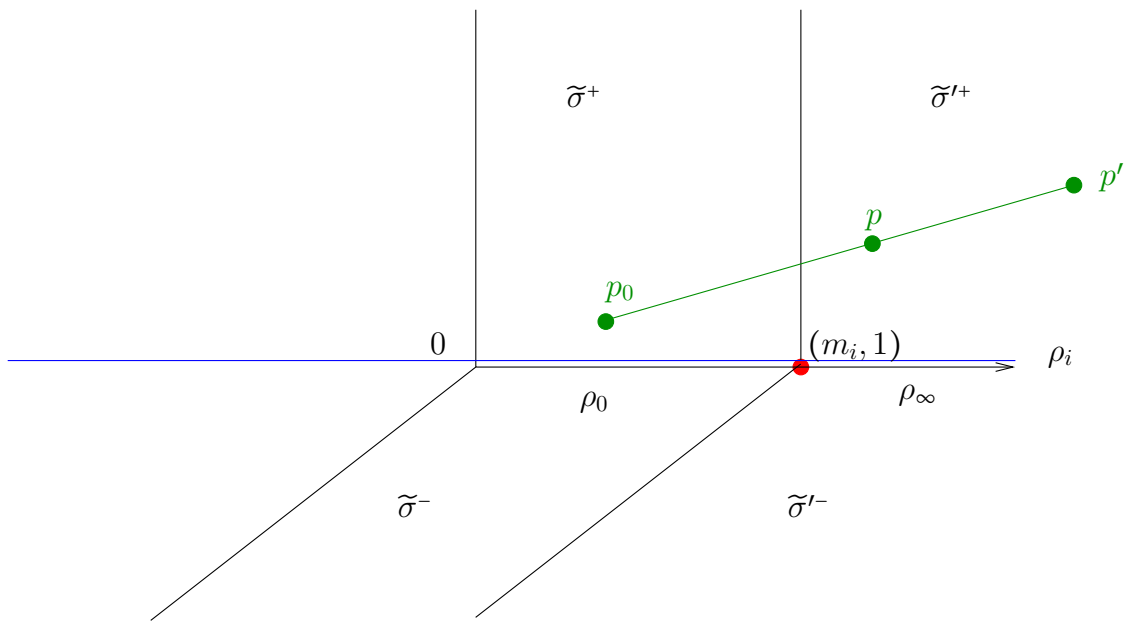}}
\caption{The points $p,p'$ and $p_0$}
\label{Fig:note}
\end{figure}
In \cite[\S4]{AG} we carried a rigorous analysis using the parallel transport map around the pieces of the discriminant locus of the wall structure $\foD^1_{}(\widetilde X, \widetilde D)$. A key point in that analysis
is a proof that this wall structure is \emph{radiant} \cite[Definition~$4.7$]{AG}. We will need the following result of \cite[Theorem~$4.22$]{AG} as a consequence of this property of the wall structure $\foD^1_{}(\widetilde X, \widetilde D)$:
\begin{theorem}
\label{Thm rays}
Let $S=(M_{\RR}\setminus\{0\})/\RR_{> 0}$ be the
sphere parameterizing rays from the origin in $\widetilde{B}$. Choose a general point $s\in S$, such that the corresponding ray $\rho_s$ does not intersect the discriminant locus. For any point $x\in \rho_s$, denote by $f_x$ be the product of all the wall crossing functions attached to walls containing $x$. Then, there are two possibilities: either $f_x$ is independent of $x$ for any $x$, or $\rho_s$ is contained in the union of two maximal cells such that if $y$ and $y'$ are two points contained in these cells, $f_y$ and $f_{y'}$ are related by the parallel transport map in \eqref{eq:parallel transport}.
\end{theorem}
Now we are ready to state our main result:
\begin{theorem}
\label{thm: heart}
The ring of theta functions defined by broken lines in $\foD_{(X,D)}^{\heart}$ is isomorphic to the coordinate ring of the mirror to $(X,D)$.
\end{theorem}

\begin{proof}
Let $\foD^{1,F_i=0}_{(\widetilde{X},\widetilde{D})}$ be the wall structure obtained from $\foD^1_{(\widetilde{X},\widetilde{D})}$ in \eqref{Eq T0 intro} by setting all fiber classes $F_i$ as in Definition \ref{def: heart} to zero. Thus, the localization of $\foD^{1,F_i=0}_{(\widetilde{X},\widetilde{D})}$ around the origin is the heart of the canonical wall structure $\foD^{\heart}_{(X,D)}$. Moreover, it follows from \cite[Proposition $3.13$]{AG}, that asymptotically $\foD^{1,F_i=0}_{(\widetilde{X},\widetilde{D})}$ is still equivalent to the canonical wall structure associate to $(X,D)$.

Let $\mathcal{R}_{X^{\vee}}$ denote the coordinate ring for the mirror to $(X,D)$.
Let $u_j$ for $1\leq j \leq \ell$ be the primitive integral directions of the rays of the polyhedral direction $\P$ of $B$. Then, the corresponding theta functions $\theta_{u_j}$ generate $\mathcal{R}_{X^{\vee}}$ as an algebra.
Indeed, the mirror of $(X,D)$ is constructed as a smoothing of a union of affine toric varieties corresponding to the cones of the tropical space associated to $(X,D)$ \cite{GHK,GSCanScat}, and the coordinate ring for the mirror $\mathcal{R}_{X^{\vee}}$ restricted to each such affine piece is generated the monomials corresponding to the primitive integral directions $u_j$ of the rays of the corresponding cone in $(B,\P)$.  

Let $p$ be a general point of $\widetilde{B}_1$, so that the ray $\RR_{\geq 0} p$ does not intersect the discriminant locus in $\widetilde{B}_1$. For every $s \in \RR_{>0}$, we denote by $\theta_{u_j}(s p)$ the theta function $\theta_{u_j}$ computed at the $s$-rescaled point $s p$ by the scattering diagram $\foD^{1,F_i=0}_{(\widetilde{X},\widetilde{D})}$. Note that for every $s, s' \in \RR_{>0}$, $\theta_{u_j}(s p)$ and $\theta_{u_j}(s' p)$ are related by parallel transport from $s$ to $s'$ by consistency of the scattering diagram $\foD^{1,F_i=0}_{(\widetilde{X},\widetilde{D})}$. 
Our goal is to show that for $0<s<<1$, the theta functions $\theta_{u_j}(s p)$ coincide with the theta functions $\theta_{u_j}$ computed by the heart of the canonical scattering diagram $\foD_{(X,D)}^{\heart}$, and that for $s>>1$, the theta functions $\theta_{u_j}(s p)$ coincide with the theta functions $\theta_{u_j}$
computed by the canonical scattering diagram $\foD_{(X,D)}$. This will imply Theorem \ref{thm: heart}.

It is enough to show that for $0<s<<1$, all broken lines contributing to $\theta_{u_j}(sp)$ only intersect walls of $\foD_{(X,D)}^{\heart}$, and that for $s>>1$, all broken lines contributing to $\theta_{u_j}(sp)$ only intersect walls of the asymptotic scattering diagram of $\foD^{1,F_i=0}_{(\widetilde{X},\widetilde{D})}$.
For that, it is enough to study how the broken lines contributing to $\theta_{u_j}(s p)$ change as a function of $s$. As long as broken lines do not pass through the discriminant locus, it follows from the radiant property of $\foD^{1,F_i=0}_{(\widetilde{X},\widetilde{D})}$ reviewed in Theorem \ref{Thm rays} that the broken lines move continuously as a function of $s$ by $s$-rescaling of the intersection points with the walls. In particular, one obtain a one-to-one correspondence between broken lines at different values of $s$ as long as no broken line passes through the discriminant locus.

It remains to study how the broken lines change when passing through the discriminant locus. This can be done by an explicit local analysis as follows. Assume that there exists a broken line $\beta$ crossing a wall in $\rho_\infty$ for some value of $s$, going to the discriminant locus for $s$ approaching a critical value $s_{\mathrm{crit}}$. Let $a_\beta z^{m_\beta}$ be the monomial attached to the linearity domain of $\beta$ just before crossing the wall. Let $\beta_{\mathrm{in}}$ be the part of the broken line $\beta$ consisting of the linearity domains before crossing the walls.
Then, using the notation of Theorem \ref{Thm rays} and Figure \ref{Fig:note}, the possible ways to continue $\beta_{\mathrm{in}}$ in a broken line after crossing the wall in $\rho_\infty$ are in one-to-one correspondence with the monomials
in 
\begin{equation}
    \label{Eq gamma}
     a_\beta z^{m_\beta} f_{y'}^{\langle m_\beta,n \rangle}   t^{\kappa_\infty \langle m_\beta,n \rangle}\,,
\end{equation}
where $f_{y'}$ denotes the wall crossing function attached to $\rho_\infty$, and by $\kappa_0$ we denote the kink of the MVPL function $\varphi_i$ defined in \eqref{Eq: varphi-i} along $\rho_{\infty}$.
Similarly, for $s<s_{\mathrm{in}}$, the possible ways to continue the $s$-rescaling of $\beta_{\mathrm{in}}$ in a broken line after crossing the wall in $\rho_0$ are in one-to-one correspondence 
with the monomials in
\begin{equation}
    \label{Eq gamma'}
 a_\beta z^{m_\beta} f_{y}^{\langle m,n \rangle}   t^{\kappa_0 \langle m_\beta,n \rangle} 
\end{equation}
where $f_{y}$ denotes the wall crossing function attached to $\rho_0$, and by $\kappa_0$ we denote the kink of the MVPL function $\varphi_i$ defined in \eqref{Eq: varphi-i} along $\rho_0$.

 It follows from \cite[Lemma~4.20]{AG} (see the final equation in the proof of \cite[Lemma~4.20]{AG} by inserting $F_i = 0$ for the fiber classes), that the functions $f_y$ and $f_{y'}$ are related by the equation
\begin{equation}
\label{Eq: relating fs}
    f_{y}=t^{-\sum_j E^j_{\ul{\rho}}}z^{\kappa^i_{\ul{\rho}} m_i}\wp(f_{y'}),
\end{equation}
where $E^j_{\ul{\rho}}$, for
$1\le j\le \kappa^i_{\ul{\rho}}$ denotes the classes
of the exceptional curves of the blow-up $D_{\rho_i}$ along $H_{ij}$, and $\wp$ is the parallel transport map defined in \eqref{eq:parallel transport}. By substituting the formula for $f_y$ given in \eqref{Eq: relating fs} to \eqref{Eq gamma'}, one can rewrite \eqref{Eq gamma'} as
\begin{equation}
    \label{Eq gamma' 2}
      a_\beta z^{m_\beta + \kappa^i_{\ul{\rho}} m_i \langle m_\beta,n \rangle} 
   t^{\kappa_0 \langle m_\beta,n \rangle -\sum_jE^j_{\ul{\rho}}\langle m_\beta,n \rangle} \wp(f_{y'})
\end{equation}
Note that kinks $\kappa_0$ and $\kappa_\infty$ are related by the formula
\begin{equation}
    \label{eq kinks change}
    \kappa_{\infty} - \kappa_0 = -\sum_{j=1}^{\kappa^i_{\ul{\rho}}} E^j_{\ul{\rho}}.
\end{equation}
by \cite[Eqn~$3.37$]{AG}. Thus, $\kappa_0 \langle m,n \rangle -\sum_jE^j_{\ul{\rho}}\langle m,n \rangle = \wp(\kappa_{\infty})$. Hence, it follows that \eqref{Eq gamma'} is obtained by applying the parallel transport $\wp$ to 
\eqref{Eq gamma}. More precisely, the parallel transport $\wp$ induces a one-to-one correspondence between the monomials in \eqref{Eq gamma} and the monomials in  \eqref{Eq gamma'}, and so we have a well-defined continuous way to deform the broken lines across the discriminant locus from $s > s_{\mathrm{crit}}$ to $s <s_{\mathrm{crit}}$, see Figure \ref{Fig: thetaschange} for an example.

Therefore, rescaling by $s\in \R_{>0}$ the intersection points with the walls, along with the above local parallel transport around the discriminant locus,  
is a well-defined way to continuously deform the broken lines contributing to $\theta_{u_j}(sp)$ as a function of $s$. As there are finitely many $1\leq j\leq l$, and finitely many broken lines contributing to 
$\theta_{u_j}(p)$, with finitely many bendings (recall that we work modulo the ideal $I$), we deduce that for $s>0$ small enough, the distance to the origin of all the intersection points of all broken lines contributing to $\theta_{u_j}(sp)$ are all strictly smaller than one for all $1\leq j\leq l$. Hence, these broken lines only intersect walls of $\foD_{(X,D)}^{\heart}$. 
Similarly, for all $s>0$ large enough, the distance to the origin of all the intersection points of all broken lines contributing to $\theta_{u_j}(sp)$ are all strictly bigger than one for all $1\leq j\leq l$.
Hence, these broken lines only intersect walls of the asymptotic scattering diagram of $\foD^{1,F_i=0}_{(\widetilde{X},\widetilde{D})}$.
\end{proof}

\begin{remark}
Note that a particular consequence of \ref{thm: heart} is that though the heart of the canonical wall structure $\foD^{\heart}(X,D)$ associated to a log Calabi--Yau pair is defined over the localization $Q_E(X,D)$ of the relevant monoid $Q(X,D)$ at exceptional curve classes, as in \eqref{eq: heart monoid}, the mirror to $(X,D)$ is nonetheless obtained as a family over $Q(X,D)$. This is indeed natural, since in the wall structure $\foD^{\heart}(X,D)$ the only occurrence of the negative powers of exceptional curve classes are on the finitely many incoming walls -- this is a particular corollary of the main result of \cite{AG} that asymptotically the wall structure $\widetilde{\foD}^1_{(\widetilde X, \widetilde D)}$ is equivalent to the canonical wall structure, in which the coefficients of the wall functions correspond to honest effective curve classes. The occurrence of negative powers of the exceptional curve classes in the finitely many incoming walls does not change the fact that the resulting ring of theta functions obtained by broken lines in $\foD^{\heart}(X,D)$ defines a family over $Q(X,D)$.
\end{remark}

\begin{example}
\label{Ex: mirrors are same}
Let $X$ be the blow-up of a non-toric point in $\PP^2$ as in Example \ref{Ex0}, for which the the height one slice of the canonical wall structure associated to the degeneration $(\widetilde{X},\widetilde{D})$ is illustrated in Figure \ref{Fig:3}. The broken lines defining theta functions with end point at a general point $p_0$ on $\foD^{\heart}_{(X,D)}$ are given by
\begin{equation}
    \nonumber
   \vartheta_{(1,0)} = x, \,\ \,\  \vartheta_{(0,1)} = y, \,\ \,\ \mathrm{and} \,\ \,\ \vartheta_{(-1,-1)} = x^{-1}y^{-1}(1+xz^{-[E]})z^{[L]} \,.
\end{equation}
On the other hand, the broken lines defining theta functions with end point at a general point $p'$ on $\foD^{\mathrm{as}}_{(\widetilde X, \widetilde D)} \cong  \foD_{(X,D)}$ are given by
\begin{equation}
\nonumber
   \vartheta'_{(1,0)} = x, \,\ \,\  \vartheta'_{(0,1)} = xy, \,\ \,\ \mathrm{and} \,\ \,\ \vartheta'_{(-1,-1)} = x^{-1}y^{-1}(1+x^{-1}z^{[E]})z^{[L-E]} \,.
\end{equation}
Observe that in this example the theta functions are related by a parallel transport map defined in \eqref{eq:parallel transport}, which is along a path on the upper half plane mapping $p_0$ to $p'$ by
\begin{align}
    \nonumber
       \wp:\mathbf{k}[\Lambda][Q(X,D)] & \longrightarrow 
\mathbf{k}[\Lambda][Q(X,D)] \\
  \nonumber
   t^qz^{(1,0)} & \longmapsto t^qz^{(1,0)} \\
   \nonumber
    t^qz^{(0,1)} & \longmapsto t^qz^{(1,1)} 
\end{align}
The mirror family to $(X,D)$ in this case, is given by
\begin{align}
   \nonumber
\mathrm{Spec}\mathbf{k}[Q(X,D)][ \vartheta_{(1,0)} , \vartheta_{(0,1)}  , \vartheta_{(-1,-1)} / \big( \vartheta_{(1,0)}  \vartheta_{(0,1)}  \vartheta_{(-1,-1)}  = z^{[L]} + \vartheta_{(1,0)} z^{[L-E]} \big) \,,
\nonumber
\end{align}
or equivalently, by
\begin{align}
   \nonumber
\mathrm{Spec} \mathbf{k}[Q(X,D)][ \vartheta'_{(1,0)} , \vartheta'_{(0,1)}  , \vartheta'_{(-1,-1)} ] / \big( \vartheta'_{(1,0)}  \vartheta'_{(0,1)}  \vartheta'_{(-1,-1)}  = z^{[L]} + \vartheta'_{(1,0)} z^{[L-E]} \big) \,.
\end{align}


\begin{figure}[h!]
\begin{center}
\resizebox{.9\linewidth}{!}{\input{thetaschange.pspdftex}}
\caption{The broken lines defining theta functions on $\foD^{\heart}_{(X,D)}$ on the left and on $\foD_{(X,D)}$ on the right}
\label{Fig: thetaschange}
\end{center}
\end{figure}

\end{example}

\section{Explicit equations for mirrors to log Calabi--Yau pairs in dimension three}
\label{Sec: examples}

In this section we first illustrate how
to obtain the concrete equation of the mirror, in the simple situation when we blow-up a single hypersurface in a three dimensional log Calabi--Yau pair. We then study the situation when several hypersurfaces are blown-up.

\begin{example}
Let $\Sigma$ be the toric fan of $X_{\Sigma}=\PP^3$, with rays generated by $e_1,e_2,e_3$ and $e_1-e_2-e_3$. Consider the blow-up of $X_{\Sigma}$ with center a general degree $d$ hypersurface $H\subset D_1$ contained in a component $D_1$ in the toric boundary corresponding to the ray generated by $e_1$. The initial walls of the heart of the associated wall structure, $\foD^{\heart}_{(\PP^3,H)}$ are displayed in Table \ref{Table: walls of P3 with one line}. 
\begin{table}[]
    \centering
    \begin{tabular}{ll} \hline
 $\fod$ & $f_{\fod}$ \\ \hline
  $\langle e_1,e_2 \rangle,\langle e_1,e_3 \rangle,\langle e_1,-e_1-e_2-e_3 \rangle$ & ~~ $(1+t^{-E}x)^d$  \\ 
 $\langle -e_1,e_2 \rangle,\langle -e_1,e_3 \rangle,\langle -e_1,-e_1-e_2-e_3 \rangle$ & ~~$(1+t^{L-E}x)^d$ \\ 
\hline
\vspace{0.0001 cm}
  \end{tabular}
    \caption{Initial walls of $\foD^{\small\heart}_{(\PP^3,H)} $, where $L$ denotes class of the strict transform of a general line in $\PP^3$ and $E$ denotes the class of a fiber of the exceptional divisor. By $\langle e_i, e_j \rangle$ we denote the cone spanned by $e_i$ and $e_j$.}
    \label{Table: walls of P3 with one line}
\end{table}
To obtain a consistent wall structure we extend each of the initial walls as illustrated in Figure \ref{Fig: P3 one line}. Let $\varphi$ be the PL function as in \S\ref{subsec:varphi}, which vanishes on the positive
octant and whose kinks on each of the two dimensional cells of $\Sigma$ are the class $[L]$ of the strict transform of a general line in $\PP^3$. We fix a general point $p$ in the positive octant.
The theta functions with endpoint $p$, and asymptotic directions given by the asymptotic directions of the rays of $\Sigma$, are given by
\begin{eqnarray}
\label{Eq:mirror one line}
     \vartheta_{e_1} & = & z^{(1,0,0)} = x, \\
     \nonumber
    \vartheta_{e_2} & = & z^{(0,1,0)} = y,   \\
    \nonumber
    \vartheta_{e_3} & = & z^{(0,0,1)} = z, \\
    \nonumber
    \vartheta_{e_4} & = & z^{(-1,-1,-1)} \big( 1+ t^{-[E]}x  \big)^d t^{[L]} =\frac{1}{\vartheta_1 \vartheta_2 \vartheta_3}\big( 1+ t^{-[E] }\vartheta_1 \big)^d t^{[L]} \,,
\end{eqnarray}
where the factor $t^{[L]}$ is the contribution from the kink of $\varphi$.
It follows from \eqref{Eq:mirror one line} that the mirror to $(X,D)$ is given by
\begin{equation}
    \label{Eq mirror for one H}
\mathrm{Spec}\mathbf{k}[Q(X,D)][ \vartheta_{e_1},\vartheta_{e_2},\vartheta_{e_3},\vartheta_{e_4}]/\big( \vartheta_{e_1}\vartheta_{e_2}\vartheta_{e_3}\vartheta_{e_4}= ( 1+ t^{-[E] }\vartheta_{e_1} )^d   t^{[L]}  \big),
\end{equation}
where $[L]$ is the class of a general line and $[E]$ is the class of an exceptional fiber over $H$, and $Q(X,D)$ is the relevant monoid associated to $(X,D)$ defined as in \eqref{eq: monoid for XD}.

\begin{figure}
\resizebox{.7\linewidth}{!}{\input{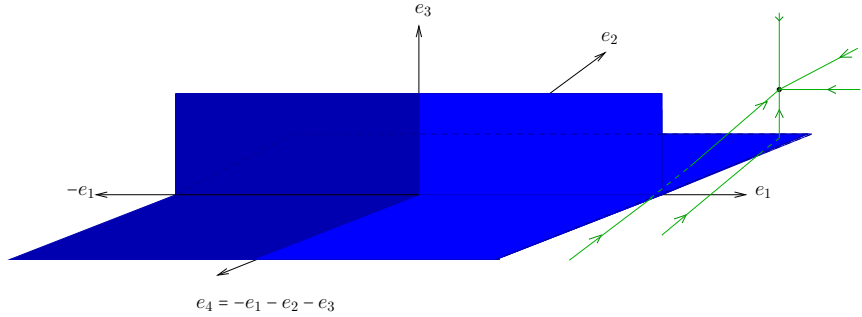}}
\caption{The initial walls of $\foD^{\heart}_{\PP^3,H}$ formed by the widget corresponding to a hypersurface of degree $d$ in the toric boundary in light blue. The consistent wall structure $\foD^{\heart}_{\PP^3,H}$ is obtained by inserting the three walls in dark blue. The broken lines defining the theta functions with endpoint in the positive octant are illustrated in green.}
\label{Fig: P3 one line}
\end{figure}
\end{example}

We next consider the situation when we blow-up a disjoint union of two hypersurfaces contained in toric boundary components of $\PP^3$. In this case, the walls formed by widgets of the tropicalizations of these hypersurfaces interact. This creates a pretty sophisticated wall structure, even in the simplest case when the center of blow up is a disjoint union of two lines, and requires to do first the combinatorial construction of the toric wall structure $\foD_{(X_{\Sigma},H)}$ for purposes of book keeping, and then passing to the heart of the canonical wall structure. Before proceeding with the more general situation, we first analyse in detail the case with two lines, in which we a priori obtain infinitely many walls in the heart of the canonical wall structure. 

\begin{example}
\label{Ex: two lines}
Let $X$ be the blow-up of $\mathbb{P}^3$ with center two disjoint lines $\ell_1,\ell_2$ contained in two different components $D_1,D_2$ in the toric boundary divisor $D_{\Sigma} \subset \PP^3$, and $D$ be the strict transform of $D_\Sigma$. 
The set of ray generators of the toric fan $\Sigma$ of $\PP^3$ is given by $\{ e_1,e_2,e_3,e_4=-e_1-e_2-e_3   \}$, where $\{ e_i ~ | ~ 1 \leq i \leq 3 \}$ is the standard basis in $\RR^3$. We further set \[z^{e_1}=x, \, \, z^{e_2}=y, \, \, z^{e_3}=z, \, \, \mathrm{and} \,\, z^{-e_1-e_2-e_3}=1/xyz.\] 
The walls of the initial wall structure $\foD_{(\PP^3,\ell_1 \cup \ell_2),\mathrm{in}}$ are formed by the two widgets, 
illustrated in Figure \ref{Fig: two widgets}. 
\begin{figure}
\resizebox{0.8\textwidth}{!}{
\input{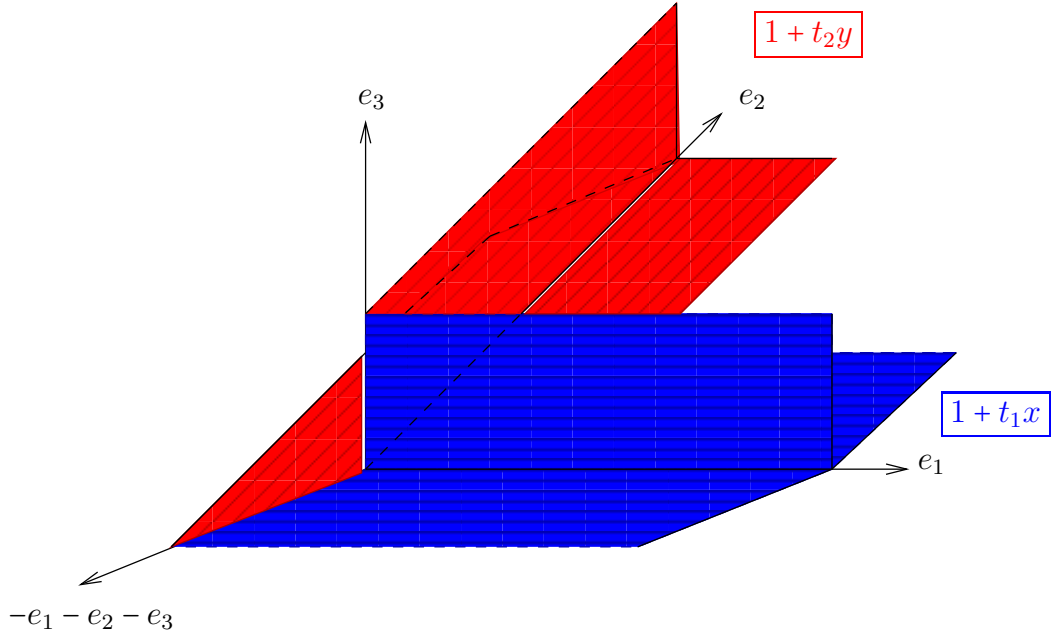}
}
\caption{The walls of $\foD_{(\PP^3,\ell_1 \cup \ell_2),\mathrm{in}}$ formed by two widgets obtained by deformations of the two tropical lines corresponding to $\ell_1$ and $\ell_2$.}
\label{Fig: two widgets}
\end{figure}
We list the set of walls of  $\foD_{(\PP^3,\ell_1 \cup \ell_2),\mathrm{in}}$ in Table \ref{Table: initial walls of P3 with two lines}.
\begin{table}[]
    \centering
    \begin{tabular}{ll} \hline
 $\fod$ & $f_{\fod}$ \\ \hline
  $\langle e_1,e_2 \rangle,\langle e_1,e_3 \rangle,\langle e_1,-e_1-e_2-e_3 \rangle$ & ~~ $1+t_1x$  \\ 
 $\langle e_2,e_1 \rangle,\langle e_2,e_3 \rangle,\langle e_2,-e_1-e_2-e_3 \rangle$ & ~~ $1+t_2y$  \\ 
\hline
\vspace{0.0001 cm}
  \end{tabular}
    \caption{Walls of $\foD_{(\PP^3,\ell),\mathrm{in}} $ formed by the two widgets in Figure \ref{Fig: two widgets}.}
    \label{Table: initial walls of P3 with two lines}
\end{table}
The set of ray generators for the initial joints in $\foD_{(\PP^3,\ell_1 \cup \ell_2),\mathrm{in}} $ is then given by
\[ \{ (-1, -1, -1) ,  
     (0, 0, 1) , 
   (1, 0, 0) , 
   (0, 1, 0)   \} \]
We first need to check for consistency up to order $1$ around all these initial joints, and then repeat it consecutively for higher orders, analysing also the new joints formed at each step. We describe how to do this in detail below.
\begin{center}
\textbf{Order $1$ Corrections}
\end{center}
Let us denote $\foD_1:=\foD_{(\PP^3,D_{\Sigma}),in}$.
First, we check consistency in $\foD_1$ around the joint generated by $ (1, 0, 0) $: The projections of the walls of $\foD_1$ adjacent to $\langle 1,0,0 \rangle$, along $\langle 1,0,0 \rangle$, are illustrated in Figure \ref{First order correction around x}. Note that to remember the normal directions of the walls adjacent to a joint, we label them on each of the rays obtained after projecting them along the joint.
\begin{figure}
\resizebox{0.8\textwidth}{!}{
\input{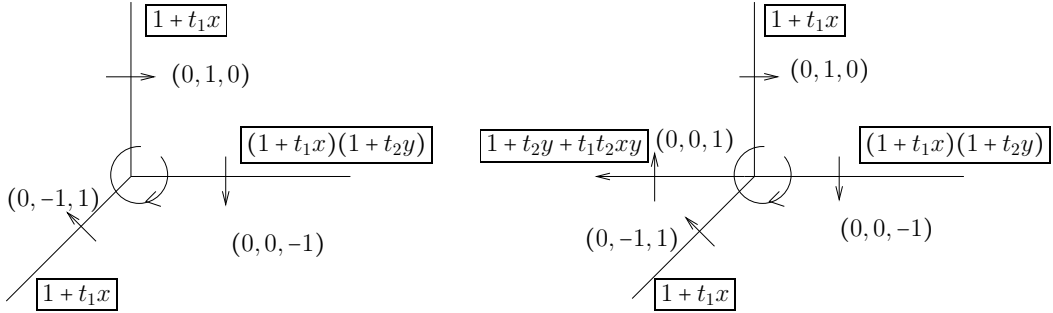}
}
\caption{On the left is the projection of the walls of $\foD_1:=\foD_{(\PP^3,D_{\Sigma}),in}$ adjacent to the joint $\langle (1,0,0) \rangle$, along $\langle (1,0,0) \rangle$. On the right is the projection of the walls of $\foD_{(\PP^3,D_{\Sigma})}$, adjacent to $\langle (1,0,0) \rangle$. We write the attached function to each wall inside the nearby box.}
\label{First order correction around x}
\end{figure}
By the formula \eqref{Eqn: theta_fop}, the wall crossing functions attached to the walls of $\foD_1$, transfer the monomials $x,y$ and $z$ as follows: $x$ remains invariant since in \eqref{Eqn: theta_fop} the power of the wall crossing function vanishes. For $y$, consecutively applying the wall crossing transformations, going counterclockwise around the joint with a loop illustrated as in Figure \ref{First order correction around x}, we obtain
\begin{align*} 
y   \mapsto y 
 \mapsto y(1+t_1x)^{-1}
  \mapsto y(1+t_1x)^{-1}(1+t_1x) = y
\end{align*} 
Hence, $y$ remains invariant as well. However for $z$, at order $1$ (i.e. up to higher order terms of degree at least $2$) we obtain, 
\begin{align*} 
z  \mapsto & z(1+t_1x)^{-1}(1+t_2y)^{-1} \\
  \mapsto  &  z(1+t_1x)(1+t_1x)^{-1}(1+t_2y(1+t_1x)^{-1})^{-1}= z(1+t_2y(1+t_1x)^{-1})^{-1} \\
  \mapsto & z(1+t_2y(1+t_1x)(1+t_1x)^{-1})^{-1}  = z(1-t_2y)
\end{align*} 
Hence at first order, $z$ is not invariant. To correct the discrepancy for $z$ to be invariant at first order, following the recipe explained in \cite[Theorem~5.6]{AG}, we set
\[\foD_2 := \foD_1 \cup (\langle (1,0,0),(0,-1,0) \rangle ,1+t_2y)\]
to be the wall structure obtained from $\foD_1$ by inserting the wall $ (\langle (1,0,0),(0,-1,0) \rangle ,1+t_2y) $. 
Next we check consistency around the joint generated by $ (0, 1, 0) $: The walls of $\foD_2$ which are adjacent to $ \langle (0, 1, 0) \rangle$ are illustrated after projecting along $ \langle (0, 1, 0) \rangle$ in Figure \ref{First order correction around y}.
\begin{figure}
\resizebox{0.8\textwidth}{!}{
\input{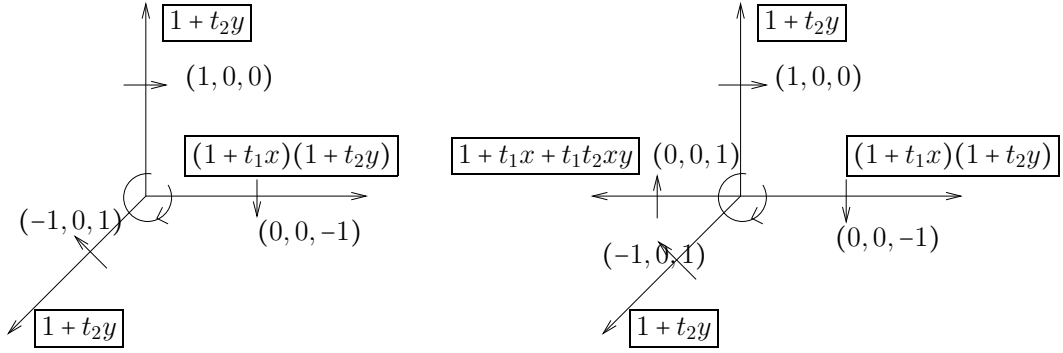}
}
\caption{On the left is the projection of the walls of $\foD_1:=\foD_{(\PP^3,D_{\Sigma}),in}$ adjacent to $\langle (0,1,0) \rangle$, along $\langle (0,1,0) \rangle$. On the right is the projection of the walls of  $\foD_{(\PP^3,D_{\Sigma})}$, adjacent to $\langle (0,1,0) \rangle$. We write the attached function to each wall inside the nearby box.}
\label{First order correction around y}
\end{figure}
Applying the wall-crossing automorphisms to $x,y,z$, at degree $1$, we obtain
\[ x\mapsto x, ~ y \mapsto y, ~ z \mapsto z(1-t_1x)  \]
Hence, to do the first order correction, we set
\[\foD_3 := \foD_2 \cup (\langle (0,1,0),(-1,0,0) \rangle ,1+t_1x)\]
Next checking consistency around the joint generated by $ (0, 0, 1) $
and proceeding analogously, we define
\[\foD_4 := \foD_3 \cup (\langle (0,0,1),(-1,0,0) \rangle ,1+t_1x) + (\langle (0,0,1),(0,-1,0) \rangle ,1+t_2y)\]
For consistency around the joint generated by $ (-1, -1, -1) $, applying the wall-crossing automorphisms to the monomials $x,y,z$, we obtain 
\[ x\mapsto x(1+t_2y), ~   y\mapsto y(1-t_1x), z \mapsto z(1+t_1x-t_2y)\]
Proceeding analogously, we set
\[ \foD_5 := \foD_4 \cup ( \langle (-1,-1,-1),(-1,0,0) \rangle ,1+t_1x) + (\langle (-1,-1,-1),(0,-1,0) \rangle , 1+t_2y\
)\]
The set of all joints of $\foD_5$ is given by
\[ \{ (0, -1,  0) ,
    (-1,  0,  0) ,
       (-1, -1, -1) ,
     (0, 0, 1) ,
       (1, 0, 0) ,
      (0, 1, 0)  \}. \]
Now, it is easy to verify that $\foD_5$ is consistent to order $1$ around all these joints. So, we can continue with consistency at order $2$.\\
\begin{center}
\textbf{Order $2$ Corrections}
\end{center}
Consistency around the joint generated by $   (-1, -1, -1) $: $\foD_5$ is not consistent to order $2$ around this joint. Indeed, the wall crossing functions transform $x,y,z$ by
\[x \mapsto x(1+t_1t_2xy), y \mapsto y(1-t_1t_2xy), z \mapsto z \]
To correct this, we define
\[ \foD_6 := \foD_5 \cup (\langle (1,0,0),(-1,-1,0) \rangle ,1+t_1t_2xy).\]
Consistency around the joint generated by $   (0, 1, 0) $: $\foD_6$ is not consistent to order $2$. The wall crossing functions transform $x,y,z$ by
\[ x \mapsto  x, y \mapsto y,  z\mapsto z(1-t_1t_2xy)   \]
To correct this, we define
\[\foD_7 := \foD_6 \cup (\langle (0,1,0),(-1,-1,0) \rangle ,1+t_1t_2xy)\]
Consistency around the joint generated by $   (0, 0, 1) $: $\foD_7$ is not consistent to order $2$.  The wall crossing functions transform $x,y,z$ by
\[ x \mapsto x(1+t_1t_2xy), y \mapsto y(1-t_1t_2xy), z \mapsto  z \]
So, we define 
\[ \foD_8 := \foD_7 \cup (\langle (0,0,1),(-1,-1,0) \rangle,1+t_1t_2xy).\]
Consistency around the joint generated by $   (-1, -1, -1) $: $\foD_8$ is not consistent to order $2$. The wall crossing functions transform $x,y,z$ by
\[ x \mapsto x(1+t_1t_2xy), y \mapsto y(1-t_1t_2xy),
 z \mapsto    z \]
To correct this, we define
\[ \foD_9 := \foD_8 \cup (\langle (-1,-1,-1),(-1,-1,0) \rangle,1+t_1t_2xy).\]
Now, we are done with order $2$. Note that the set of joints of $\foD_9$ is given by
\[ \{
     (0, -1,  0) ,
      (-1,  0,  0)  ,
      (-1, -1,  0)  ,
        (-1, -1, -1)  ,
        (0, 0, 1)  ,
     (1, 0, 0)  ,
       (0, 1, 0)  
\} \]
and $\foD_9$ is consistent around all these joints to order $2$. Moreover, it is easy to verify that in $\foD_9$  around the joints $(0,0,1)$ and $(-1,-1,-1)$ we already have consistency to all orders, hence no new wall which are adjacent to either of $z$ or $1/xyz$ will be inserted at the next steps. Although the process of inserting new walls will never terminate in this example, all the remaining walls will have support on the plane spanned by $e_1$ and $e_2$. Using magma computer algebra, and continuing to do higher order corrections around the other joints consecutively we deduce that achieving consistency around the joint $ \langle (1, 0, 0) \rangle $ requires the insertion of infinitely many new walls to $\foD_9$, given by
\[ (\langle (1,0,0),(-1,-1,0) \rangle ,1+t_1t_2xy) \cup \bigcup_{\substack{(a,b)\in \ZZ^2\\b<a<0}} (\langle (1,0,0),(a,b,0) \rangle ,f_{(a,b,0)})\,.\]
To write the equations of mirror families, we do not need to provide closed formulas for $f_{(a,b,0)}$'s, as we will see in a moment. We nonetheless note that such a closed formula would provide one the data of of counts of $\AA^1$-curves -- as explained in \cite[\S~7]{AG}, such counts correspond to coefficients of $\log f_{(a,b,0)}$. It is a challenging task beyond the scope of this paper to write such closed formulas. 
Proceeding similarly, achieving consistency around the other joints of $\foD_9$,
requires the insertion of infinitely many new walls to $\foD_9$ with support on the plane spanned by $e_1$ and $e_2$. We compute the limits of all the products of the corresponding wall crossing functions and obtain the following:
\begin{proposition}
\label{x-y}
The walls of $\foD_{(\PP^3,\ell_1 \cup \ell_2)}$, up to equivalence, are displayed in
Table \ref{Table: walls of final HDTV P3}. 
\end{proposition}
\begin{proof}
Since up to equivalence there is a unique consistent will structure, is 
suffices to check that the wall structure with the final walls listed in Table \ref{Table: walls of final HDTV P3} is consistent. For this, we check consistency around each of the joints. Tracing around the joint $\langle (1,0,0) \rangle $, with a loop illustrated on the right hand side of Figure \ref{First order correction around x}, we immediately obtain $x \mapsto x$, $y\mapsto y$ as computed above while doing the first order corrections on $\foD_1$. Moreover, now for the monomial $z$, we obtain 
\begin{align*} 
z  \mapsto & z(1+t_1x)^{-1}(1+t_2y)^{-1} \\
  \mapsto  & z(1+t_1x)(1+t_1x)^{-1}(1+t_2y(1+t_1x))^{-1} = z(1+t_2y(1+t_1x))^{-1} \\
  \mapsto & z(1+t_2y+t_1t_2xy)(1+t_2y(1+t_1x))^{-1} = z
  \mapsto z
\end{align*}
Hence, we get consistency to all orders around $\langle (1,0,0) \rangle $.
Consistency around  $\langle (0,1,0) \rangle $ follows analogously by replacing $x$ by $y$ in the above computation. The consistency around the other joints is an analogous straight forward computation.
\end{proof}
We illustrate the walls of $\foD_{(\PP^3,\ell_1\cup \ell_2)}$ in Figure \ref{Walls lines}.
\begin{figure}
\resizebox{0.7\textwidth}{!}{
\input{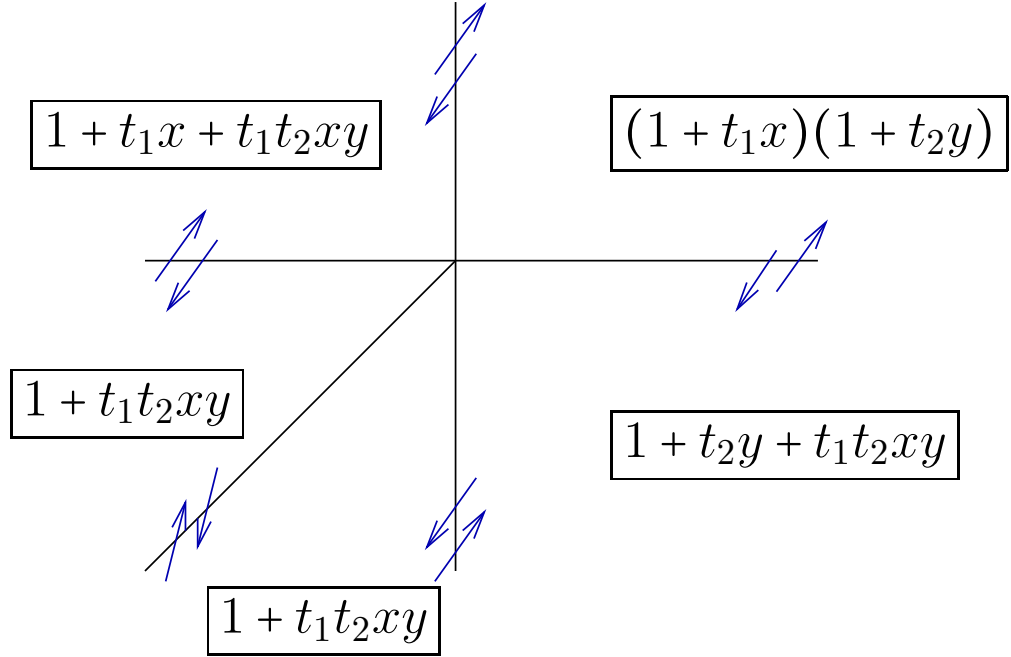}
}
\caption{Walls of the consistent wall structure $\foD_{(\PP^3,\ell_1\cup \ell_2)}$ which lie on the $\langle e_1,e_2 \rangle$ plane. Each upward pointing arrow on a joint indicates that there is a wall spanned by it and $\langle (1,0,0) \rangle$. Each downward pointing arrow on a joint indicates that there is a wall spanned by it and $\langle (-1,-1,-1) \rangle$. }
\label{Walls lines}
\end{figure}
We pass from the $t_i$-variables on the walls of $\foD_{(\PP^3,\ell_1 \cup \ell_2)}$ to the
curve classes variables, as explained in 
Section \ref{Sec: Degeneration}. Further, we insert all fiber classes $F=0$, as explained in \S\ref{Sec: the heart of the canonical wall structure}, to obtain the walls of the heart $\foD^{\small\heart}_{(X,D)}$ of the canonical wall structure, which are displayed in Table
\ref{Table: final walls canonical P3}.
\begin{table}[]
    \centering
    \begin{tabular}{ll} \hline
 $\fod$ & $f_{\fod}$ \\ \hline
  $\langle e_1,e_2 \rangle,\langle e_1,e_3 \rangle,\langle e_1,e_4 \rangle$ & ~~ $1+t_1x$  \\ 
$\langle e_2,e_1 \rangle,\langle e_2,e_3 \rangle,\langle e_2,e_4 \rangle$ & ~~$1+t_2y$ \\ 
  $\langle e_3,-e_1 \rangle,\langle e_4,-e_1 \rangle$ & ~~ $1+t_1x$  \\  
  $\langle e_3,-e_2 \rangle,\langle e_4,-e_2 \rangle$ & ~~ $1+t_2y$  \\ 
 $\langle -e_2,-e_1-e_2 \rangle,\langle -e_1,-e_1-e_2 \rangle,\langle e_3,-e_1-e_2 \rangle,\langle e_4,-e_1-e_2 \rangle$ & ~~$1+t_1t_2xy$ \\
 $\langle e_1,-e_2 \rangle$ & ~~ $1+t_2y+t_1t_2xy$  \\ 
  $\langle e_2,-e_1 \rangle$ & ~~ $1+t_1x+t_1t_2xy$  \\ 
\hline
  \end{tabular}
   \caption{Walls of $\foD_{(\PP^3,\ell_1\cup \ell_2)} $, where $e_4 = -e_1-e_2-e_3$. Here the first two rows correspond to initial walls.}
    \label{Table: walls of final HDTV P3}
\end{table}
\begin{table}[ht]
 \begin{tabular}{ll} \hline
 $\fod$ & $f_{\fod}$ \\ \hline
 $\langle e_1,e_2 \rangle,\langle e_1,e_3 \rangle, \langle e_1,e_4 \rangle$ & ~~ $1+t^{-E_1}x$  \\ 
$\langle e_2,e_1 \rangle,\langle e_2,e_3 \rangle, \langle e_2,e_4 \rangle$ & ~~$1+t^{-E_2}y$ \\
 $\langle e_3,-e_1 \rangle, \langle e_4,-e_1 \rangle$ & ~~ $1+t^{L-E_1}x$  \\
 $\langle e_3,-e_2 \rangle,\langle e_4,-e_2 \rangle,$ & ~~ $1+t^{L-E_2}y$  \\
  $\langle -e_1,-e_1-e_2 \rangle,\langle -e_2,-e_1-e_2 \rangle,  \langle e_3,-e_1-e_2 \rangle,  \langle e_4,-e_1-e_2 \rangle$ & ~~$1+t^{L-E_1-E_2}xy$ \\
 $\langle e_1,-e_2 \rangle$ & ~~ $1+t^{L-E_2}y+t^{L-E_1-E_2}xy$  \\ 
  $\langle e_2,-e_1 \rangle$ & ~~ $1+t^{L-E_1}x+t^{L-E_1-E_2}xy$  \\ 
\hline
  \end{tabular}
  \caption{Walls of $\foD^{\small\heart}_{(\mathrm{Bl}_{\ell_1\cup \ell_2}(\PP^3),D)}$}
    \label{Table: final walls canonical P3}
  \end{table}
 Observe that, by picking a general point $p$ in the positive octant spanned by $e_1,e_2$, and $e_3$ we ensure that the only broken lines that are not never-bending are the ones with asymptotic direction $-e_1-e_2-e_3$, which cross the wall $( \langle (1,0,0),(0,1,0) \rangle, (1+t^{[-E_1]}x)(1+t^{[-E_2]}y))$. The theta functions defined by these broken lines are given by
 \begin{equation}
  \label{Eq: theta for line}
     \vartheta_{e_1}  = x, ~~   ~~   \vartheta_{e_2}  = y, ~~  ~~ \vartheta_{e_3}  = z, ~~ ~~ \mathrm{and} ~~  ~~ \vartheta_{e_4}  =  x^{-1}y^{-1}z^{-1}(1+t^{[-E_1]}x)(1+t^{[-E_2]}y)t^{[L]}. 
 \end{equation}
In this case, the mirror to $(X,D)$ is given by
\[ \mathrm{Spec}\mathbf{k}[Q(X,D)][ \vartheta_{e_1},\vartheta_{e_2}, \vartheta_{e_3},\vartheta_{e_4}]/(\vartheta_{e_1}\vartheta_{e_2} \vartheta_{e_3}\vartheta_{e_4}=(1+t^{[-E_1]} \vartheta_{e_1})(1+t^{[-E_2]} \vartheta_{e_2})t^{[L]}), \]
where $-[E_1]$ denotes the class of an exceptional curve over $\ell_1$, and $-[E_2]$ is the class of an exceptional curve over $\ell_2$. 
 
  \end{example}

 \begin{remark}
 \label{Rem CC}
  Note that $X$ in Example \ref{Ex: two lines} 
  is a Fano variety with Mori-Mukai name MM $3-25$ 
  -- see \cite{coates2016quantum}. Moreover, the associated superpotential to $X$ given in \cite[Table~1]{coates2016quantum}, agrees with the sum of the theta functions we compute in \eqref{Eq: theta for line}, defining the tropical superpotential as in \cite{CPS}, which conjecturally agrees with the superpotential of \cite{coates2016quantum}.
So, we verify that the expectation that the mirror construction of \cite{GSCanScat} is compatible with the manifestations of Landau--Ginzburg mirror symmetry.
  \end{remark}

\begin{example}
\label{Ex: general case}
Let $X$ be the blow-up of $\mathbb{P}^3$ with center a disjoint union of a degree $d_1$ hypersurface $H_1$ and a degree $d_2$ hypersurface $H_2$ contained in two different components, say $D_1$ and $D_2$ respectively, in the toric boundary divisor $D_{\Sigma} \subset \PP^3$, and $D$ be the strict transform of $D_\Sigma$. We list the set of walls of  $\foD_{(\PP^3,H_1 \cup H_2),\mathrm{in}}$, along with the attached functions in this case in Table \ref{Table: initial walls of P3 with line and conic}.
\begin{table}[]
    \centering
    \begin{tabular}{ll} \hline
 $\fod$ & $f_{\fod}$ \\ \hline
  $\langle e_1,e_2 \rangle,\langle e_1,e_3 \rangle,\langle e_1,-e_1-e_2-e_3 \rangle$ & ~~ $(1+t_1x)^{d_1}$  \\ 
 $\langle e_2,e_1 \rangle,\langle e_2,e_3 \rangle,\langle e_2,-e_1-e_2-e_3 \rangle$ & ~~ $(1+t_2y)^{d_2}$  \\ 
\hline
\vspace{0.0001 cm}
  \end{tabular}
    \caption{Walls of $\foD_{(\PP^3,H_1 \cup H_2),\mathrm{in}} $}
    \label{Table: initial walls of P3 with line and conic}
\end{table}
Doing order by order consistency check around all joints, analogously as in \S\ref{Ex: two lines} we obtain infinitely many walls on the plane spanned by $e_1$ and $e_2$. By the aid of magma computer algebra \cite{Magma}, we deduce that the final consistent wall structure again is formed by these infinitely many walls supported on the $\langle e_1,e_2 \rangle$ plane, together with walls whose supports are on the cones $\langle j ,e_3 \rangle $ and $\langle j ,e_1-e_2-e_3 \rangle $, for any joint $j$ on the $e_1-e_2$ plane displayed in Figure \ref{Walls general}.
\begin{figure}
\resizebox{0.7\textwidth}{!}{
\input{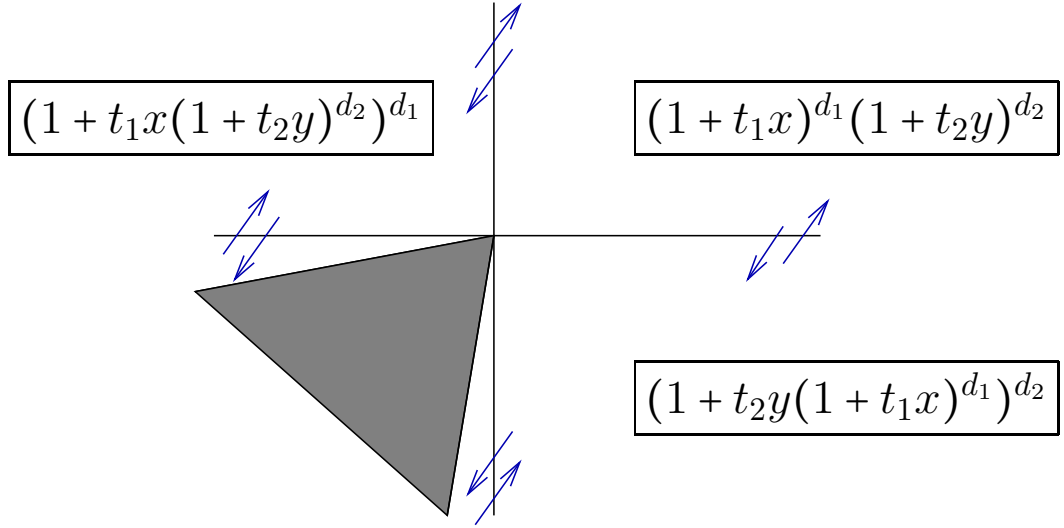}
}
\caption{Walls of the consistent wall structure $\foD_{(\PP^3,H_1\cup H_2)}$ which lie on the $\langle e_1,e_2 \rangle$ plane. Each upward pointing arrow on a joint indicates that there is a wall spanned by it and $\langle (1,0,0) \rangle$. Each downward pointing arrow on a joint indicates that there is a wall spanned by it and $\langle (-1,-1,-1) \rangle$. }
\label{Walls general}
\end{figure}
Again by picking a generic point $p$ in the positive octant spanned by $e_1,e_2,e_3$, similarly as in \S\ref{Ex: two lines}, we the obtain the theta functions defined by broken lines with endpoint $p$, given by 
\begin{equation}
    \label{Eq:line and conic}
     \vartheta_{e_1}  = x, ~~   ~~   \vartheta_{e_2}  = y, ~~  ~~ \vartheta_{e_3}  = z, ~~ ~~ \mathrm{and} ~~  ~~ \vartheta_{e_4}  = x^{-1}y^{-1}z^{-1}(1+t^{[-E_1]}x)^{d_1}(1+t^{[-E_2]}y)^{d_2}t^{[L]}.
\end{equation}
Therefore, the mirror to $(X,D)$ is given by
\begin{equation}
\label{Eq: mirror to general}
    \mathrm{Spec}\mathbf{k}[Q(X,D)][ \vartheta_{e_1},\vartheta_{e_2}, \vartheta_{e_3},\vartheta_{e_4}]/(\vartheta_{e_1}\vartheta_{e_2} \vartheta_{e_3}\vartheta_{e_4}=(1+t^{[-E_1]} \vartheta_{e_1})^{d_1}(1+t^{[-E_2]} \vartheta_{e_2})^{d_2}t^{[L]}), 
\end{equation}
where $Q(X,D)$ is the relevant monoid associated to $(X,D)$ defined as in \eqref{eq: monoid for XD}, $[L]$ is the class of a general line, $[E_1]$ is the class of a fiber over $H_1$ and  $[E_2]$ is the class of a fiber over $H_2$.
\end{example}

\section{Comparison with the work of Abouzaid--Auroux--Katzarkov}

In this section we first overview the mirror construction of Abouzaid--Auroux--Katzarkov for blow-ups of toric varieties along a smooth hypersurface \cite{AAK} using symplectic geometric techniques and then compare it with our construction \cite{AG} following the algebro-geometric framework of Gross--Siebert. The main result in this section shows that these two constructions agree.

Let $V$ be a smooth projective toric variety and $H\subset V$ a smooth hypersurface. Denote by $X$ the blow-up of $\PP^1 \times V$ at $\{ 0 \} \times H $, and let $D$ be the strict transform of the toric boundary divisor of $\PP^1 \times V$. The mirror to the log Calabi--Yau pair $(X,D)$, by which we mean the mirror to the open Calabi--Yau manifold $X^0=X \setminus D$, from the SYZ point of view \cite{SYZ} is constructed by Abouzaid--Auroux--Katzarkov using symplectic geometric techniques \cite{AAK}.\footnote{In \cite{AAK} one starts with $\CC \times V$ rather than $\PP^1 \times V$. However, the complement $X\setminus D$ in either case if the same. Thus, for convenience in this section we adopt \cite{AAK} to the situation when we start with the compact toric variety $\PP^1 \times V$ apriori, to be able to compare it with the construction of \cite{AG}.} Note that $X^0$ is apriori described on \cite[pg 5]{AAK} as a conic bundle. However, it follows that it actually agrees with $X\setminus D$ -- see \cite[pg 16]{AAK}.

The explicit mirror construction we outline in this paper, following our work with Mark Gross using algebro-geometric tools coming from the Gross--Siebert program \cite{AG}, is in some sense both more general and in other both more special: it is more general that we can consider more than a single hypersurface, and construct mirrors to blow-ups of toric varieties along unions of many hypersurfaces. However, it is also more special as we fix the tropical types of hypersurfaces, so that the tropicalizations of the hypersurfaces we consider correspond to widgets as defined in \eqref{eq_initial}. On the other hand, in \cite{AAK}, it is allowed to consider any generic tropical type of hypersurfaces, and in particular, the mirror constructed in \cite{AAK} depends on a chosen tropical type of the hypersurface, while in \cite{AG} we apriori fix the type. In what follows, we consider the special case of the \cite{AAK} mirror where the tropical type is fixed as in \cite{AG}.

\begin{remark}
\label{Remark constant family}
Allowing the tropical type of the hypersurface to vary as in \cite{AAK}, amounts to considering a $1$-parameter of hypersurfaces $H_t$ inside $V$, which define a $1$-parameter family of complex structures on $X\setminus D$. As manifested by mirror symmetry, the complex moduli space corresponds to the K\"ahler moduli space of the mirror. Hence, the mirror of \cite{AAK} encodes the data of a choise of K\"ahler parameter, while in the construction of \cite{AG} this parameter is fixed. A particular consequence of fixing such a parameter is that the mirrors to the blow ups of toric varieties in \cite{AG}, which we explicitly write equations for in this paper, are typically singular. However, in \cite{AAK}, by varying the K\"ahler parameter, which amounts to a birational modification of the mirror, they construct a smooth mirror. Nonetheless, their construction can be carried in the situation when one considers a constant family of hypersurfaces, and in this special case we show it agrees with our construction. We expect that one can construct a $1$-parameter family of the (heart of the) canonical wall structure as in \cite{GSCanScat}, which would allow one to vary the type of hypersurfaces and produce a resolution of the mirror in the situation where one works with blow ups of toric varieties along several hypersurfaces. 
\begin{flushright}
$\Box$
\end{flushright}
\end{remark}

To define the mirror to $X^0$, denoted by $Y^0$ in \cite[Thm 1.7]{AAK}, we will first describe a toric variety $Y$ by defining its momentum polytope as the upper convex hull of a piecewise-linear (PL) function. 

Assume that dim$V=n$, and let $\Sigma_V$ denote the fan of $V$ in $\RR^n$. Let 
\begin{equation}
\label{Eq: PL fnc}
    \varphi_H: \RR^n \lra \RR
\end{equation}
be a PL function with kink $H \cdot C_{\tau}$, the intersection number of $H$ with $C_{\tau}$, across a codimension one cone $\tau$ of $\Sigma_V$ where $C_{\tau}$ is the curve in $V$ corresponding to $\tau$. As discussed in \S \ref{Sec: mirrors to toric} knowing the kinks along codimension one cones, determines a PL function only up to a linear function. To get a unique PL function, without loss of generality in what follows we assume that $\varphi_H$ is zero on a given maximal dimensional cone $\sigma_0$ of $\Sigma_V$.

The PL function $\varphi_H$ is one of the main ingredients to construct  the mirror family to blow-ups of toric varieties along hypersurfaces following our work with Mark Gross -- note that $\varphi_H$ is denoted by $\varphi_i$ in \cite[Equation $3.14$]{AG}, as in that context when we consider more than one hypersurface we keep track of them by indexing with $i$. Though apriori in \cite{AG} we use an alternative description for this function, it is shown in the proof of \cite[Theorem $3.4$]{AG} that it follows from standard toric geometry that the kinks of $\varphi_i$ agree with the kinks of $\varphi_i$ described as above, given by the intersection number of $H$ with $C_{\tau}$, across a codimension one cone $\tau$ of $\Sigma_V$.

The following proposition shows that the PL function $\varphi_H$ furthermore agrees with the PL function used in the work of Abouzaid--Aroux--Katzarkov defined in \cite[Eqgn. $3.2$]{AAK}, in the particular situation when one considers a constant family of hypersurfaces, as discussed in Remark \ref{Remark constant family} (in this case the $\rho(\alpha)$ in \cite[Eqgn. $3.2$]{AAK} are all zero). 

\begin{proposition}
Let $A$ be the set of vertices of the momentum polytope image, $\Delta_V$, of the toric variety $V$, defined using the polarization defined by the hypersurface $H\subset V$. Let $\varphi: \mathrm{Support}(\Sigma_V) = \RR^n \to \RR$ be a PL function defined by
\begin{equation}
    \varphi(\xi) = \mathrm{max} \{ \langle a, \xi \rangle ~~ | ~~ a\in A \},
\end{equation}
Then, $\varphi$ agrees with $\varphi_H$ up to a linear function. 
\end{proposition}

\begin{proof}
It suffices to show that $ \varphi$ and $ \varphi_H$ have the same kinks along codimension one cones of $\Sigma_V$. Let $\tau$ be such a cone, adjacent to maximal cones $\sigma_1,\sigma_2$ of $\Sigma_V$ and let $C_{\tau}$ be the corresponding curve in $V$. It follows directly from the definition of the dual fan $\Sigma_V$ associated to $\Delta_V$, that the restriction of the PL function $\varphi$ to the maximal cones are given by the linear functions defined by
\[    \varphi_{|\sigma_1}  = \langle  \cdot , a_1  \rangle  \,\ \,\   ~~~~~~~~~~ \mathrm{and}  ~~~~~~~~~~ \,\ \,\ \varphi_{|\sigma_2}  = \langle  \cdot , a_2  \rangle  \]
where $a_1$ and $a_2$ are vertices of $\Delta_V$ corresponding to the maximal cones $\sigma_1$ and $\sigma_2$ of $\Sigma_V$ respectively. The kink of the PL function $\varphi$ along $\tau$, which by definition is the difference of the slopes of $ \varphi_{|\sigma_1}$ and $ \varphi_{|\sigma_2}$, equals the integral length of the edge with vertices $\sigma_1$ and $\sigma_2$ in $\Delta_V$. However, by standard toric geometry this integral length equals the intersection number $H \cdot C_{\tau}$. Hence, the result follows.
\end{proof}

Without loss of generality we can assume that $\varphi_H$ is zero on the maximal cone $\sigma_0$ of $\Sigma_V$, as we had assumed for $\varphi_H$. Hence, we identify the two PL functions $\varphi$ and $\varphi_H$ in the remaining part of this section.

Now, to define the mirror of $X^0$ following \cite{AAK}, one first defines the $(n+1)$-dimensional toric variety $Y$ with momentum polyope
\[ \Delta_Y = \{(\eta,\xi)\in \R \oplus \R^n\,|\, \eta \geq \varphi_H(\xi) \} \subset \R \oplus \R^n \,.\]
Let $m_1,\dots, m_r$ be primitive generators of the rays of 
$\Sigma_V$. For every $1 \leq i\leq r$, the point 
$(\varphi_H(m_i),m_i)\in \R \oplus \R^n$ belongs to 
$\Delta_Y$, and so the monomial 
$z^{(\varphi_H(m_i),m_i)}$ defines a global function $v_i$ on $Y$. 
Similarly, as $(1,0)\in \Delta_Y$, the monomial $z^{(1,0)}$ defines a global function $v_0$ on $Y$. The ring 
$\C[Y]$ of regular functions on $Y$ is generated by $v_0, v_1,\dots, v_r$,
because the vectors $(\varphi_H(m_i),m_i)$ and $(1,0)$ span the cone $\Delta_Y$.
The AAK mirror is the variety 
\begin{equation}
    \label{Eq AAK mirror}
 Y^{\circ} := Y \times \mathrm{Spec}\CC[t^{\pm E}] \setminus w_0^{-1}(0) 
\end{equation}
obtained from $Y$ by removing the hypersurface defined by the vanishing of the function $w_0:  Y \times \mathrm{Spec}\CC[t^{\pm E}]  \to \CC$ given by
\[ w_0 := -t^E+t^E v_0\,\] 
where $E$ is the class of exceptional 
$\PP^1$-fibers over $H$.

Note that, we view the mirror to the log Calabi--Yau $(X,D)$ as a family, where the complex structure can vary, hence we call the family $Y^0$ over $\CC$ defined above the mirror to $Y$, although it is natural to call a general fiber of $Y^0$ the mirror (in \cite{AAK}, $Y^0$ stands for a general fiber of the total space $Y^0 \to \CC^*$ we define above). The following main result of this sectin shows that the mirror of \cite{AAK} agrees with our mirror in \cite{AG}, when we consider the specific situation of blow ups of toric varieties along a single hypersurface.

\begin{theorem}
\label{Thm AAK agrees AG}
The restriction of the mirror family $Y\to \mathrm{Spec}\mathbf{k}[Q(X,D)]$, constructed following \cite{AG}, to the locus $\CC^* = \mathrm{Spec}\CC[t^{\pm E}] \subset \mathrm{Spec}\mathbf{k}[Q(X,D)]$ is isomorphic to the AAK mirror $Y^0$.
\end{theorem}

\begin{proof}
We first describe $\C[Y^\circ]$ as a subalgebra of the field $\C(\Z \oplus \Z^n)$ of rational functions in the monomials $z^m$ with 
$m \in \Z \oplus \Z^n$.
Let $\varphi$ be the PL function as in \eqref{Eq: PL fnc} and let $p$ be a general point in the maximal cone 
$\sigma_0$ of $\Sigma_V$ where $\varphi_H=0$. 
Fix $1 \leq i\leq r$. The line $p+\R_{\geq 0}m_i$ intersects some number (possibly zero) of codimension one cones 
$\tau_j$ of $\Sigma_V$, with normal vectors $n_{\tau_j}$. As $\varphi_H$ has a kink $\kappa_{\tau_j}$
across each of these cones, we have 
\[ \varphi_H(m_i)=\sum_j (n_{\tau_j},m_i) \kappa_{\tau_j}\]
and so 
\begin{equation} 
\label{eq:AAK} 
v_i=z^{(\varphi_H(m_i),m_i)}=z^{(0,m_i)} \prod_j (z^{(1,0)})^{(n_{\tau_j},m_i) \kappa_{\tau_j}} \,.\end{equation}
In other words, the algebra $\C[Y^{\circ}]$ is the subalgebra of 
$\C(\Z \oplus \Z^n)$ generated by
\[ v_0=z^{(1,0)}\,,\]
\[ v_i=z^{(0,m_i)} \prod_j (z^{(1,0)})^{(n_{\tau_j},m_i) \kappa_{\tau_j}} \]
for $1 \leq i \leq r$, and 
\[ w_0' = (-t^E+t^E z^{(1,0)})^{-1} \,.\]

Next, we describe the mirror following our work \cite{AG}, using the heart of the canonical wall structure we introduced in \S\ref{Sec: the heart of the canonical wall structure}, and compute its restriction to $\CC^*=\mathrm{Spec}\CC[t^{\pm E}]$ setting all non-exceptional curve classes to zero.
In particular, all the kinks of the heart of the canonical wall structure are trivial because they are all pullback of toric curve classes.
The heart of the canonical wall structure of
$(X,D)$ lives in $\R \oplus \R^n$. For every codimension-one cone 
$\tau$ of $\Sigma_V$, we have a wall $\rho_\tau :=\R \oplus \tau$ in 
$\R \oplus \R^n$, with attached function 
\[ f_{\rho_\tau}:= (1+t^{-E}z^{-(1,0)})^{\kappa_\tau}\]
where $\kappa_\tau$ is the kink of $\varphi_H$ across $\tau$.
The ring $\mathcal{R}_{X,D}$ of regular functions on the GS/HDTV mirror of $(X,D)$
is spanned by theta functions $\vartheta_0, \vartheta_1, \dots, \vartheta_r, \vartheta_{0}'$
corresponding respectively to the rays of the fan of $\PP^1 \times V$ in 
$\R \oplus \R^n$ spanned by $(-1,0)$, $(0,m_1), \dots, (0,m_r)$, $(1,0)$.

We compute the theta function at a point 
$(\epsilon, p)\in \R \oplus \R^n$ with $\epsilon \neq 0$.
First of all, we have $\vartheta_0=z^{(-1,0)}$ and 
$\vartheta_0'=z^{(1,0)}$.
For every $1 \leq i \leq r$,
moving along the line 
$(\epsilon,p) + \R_{\geq 0}(0,m_i)$, we encounter the walls $\R \oplus \tau_j$, where the cones 
$\tau_j$ are as above. In particular, we have 
\[ \vartheta_i = z^{(0,m_i)} \prod_j (1+t^{-E} z^{-(1,0)})^{(n_{\tau_j},m_i) \kappa_{\tau_j}} \,.\]
In other words, the algebra $\mathcal{R}_{(X,D)}$
is the subalgebra of 
$\C(\Z \oplus \Z^n)$ generated by 
\[ \vartheta_0=z^{-(1,0)}\,,\]
\[ \vartheta_i = z^{(0,m_i)} \prod_j (1+t^{-E} z^{-(1,0)})^{(n_{\tau_j},m_i) \kappa_{\tau_j}} \]
for $1 \leq i \leq r$, and 
\[ \vartheta_0'=z^{(1,0)}\,.\]
Comparing the embeddings of $\C[Y^{\circ}]$ and 
$\mathcal{R}_{X,D}$ in $\C(\Z \oplus \Z^n)$, we obtain that the automorphism of 
$\C(\Z \oplus \Z^n)$ defined by 
$z^{(1,0)} \mapsto 1+t^{-E}z^{-(1,0)}$ and 
$z^{(0,m)} \mapsto z^{(0,m)}$ for every 
$m \in \Z^n$ restricts to an algebra isomorphism
\[ \Psi \colon \C[Y^0] \longrightarrow \mathcal{R}_{(X,D)}
\,,\]
such that $\Psi(v_0):= 1+t^{-E} \vartheta_0$,  $\Psi(v_i)=\vartheta_i$
for $1 \leq i\leq r$, and $\Psi(w_0')=\vartheta_0'$.
\end{proof}

\bibliographystyle{plain}
\bibliography{bibliography}
\end{document}